\pdfoutput=1
\RequirePackage{ifpdf}
\ifpdf 
\documentclass[pdftex]{sigma}
\else
\documentclass{sigma}
\fi

\usepackage{mathtools}
\usepackage{microtype}
\usepackage{enumitem}
\usepackage{tikz-cd}

\numberwithin{equation}{section}

\newtheorem{Theorem}{Theorem}[section]
\newtheorem{Lemma}[Theorem]{Lemma}

 { \theoremstyle{definition}
\newtheorem{Definition}[Theorem]{Definition}
\newtheorem{Example}[Theorem]{Example}
\newtheorem{Remark}[Theorem]{Remark} }

\begin{document}

\allowdisplaybreaks

\newcommand{\arXivNumber}{1511.00061}

\renewcommand{\thefootnote}{}

\renewcommand{\PaperNumber}{019}

\FirstPageHeading

\ShortArticleName{Lagrangian Mechanics and Reduction on Fibered Manifolds}

\ArticleName{Lagrangian Mechanics and Reduction\\on Fibered Manifolds\footnote{This paper is a~contribution to the Special Issue ``Gone Fishing''. The full collection is available at \href{http://www.emis.de/journals/SIGMA/gone-fishing2016.html}{http://www.emis.de/journals/SIGMA/gone-fishing2016.html}}}

\Author{Songhao LI, Ari STERN and Xiang TANG}

\AuthorNameForHeading{S.~Li, A.~Stern and X.~Tang}

\Address{Department of Mathematics, Washington University in St.~Louis,\\
One Brookings Drive, St.~Louis MO 63130-4899, USA}
\Email{\href{mailto:lisonghao@gmail.com}{lisonghao@gmail.com}, \href{mailto:stern@wustl.edu}{stern@wustl.edu}, \href{mailto:xtang@wustl.edu}{xtang@wustl.edu}}

\ArticleDates{Received October 05, 2016, in f\/inal form March 13, 2017; Published online March 22, 2017}

\Abstract{This paper develops a generalized formulation of Lagrangian mechanics on f\/ibered manifolds, together with a reduction theory for symmetries corresponding to Lie groupoid actions. As special cases, this theory includes not only Lagrangian reduction (including reduction by stages) for Lie group actions, but also classical Routh reduction, which we show is naturally posed in this f\/ibered setting. Along the way, we also develop some new results for Lagrangian mechanics on Lie algebroids, most notably a new, coordinate-free formulation of the equations of motion. Finally, we extend the foregoing to include f\/ibered and Lie algebroid generalizations of the Hamilton--Pontryagin principle of Yoshimura and Marsden, along with the associated reduction theory.}


\Classification{70G45; 53D17; 37J15}

\renewcommand{\thefootnote}{\arabic{footnote}}
\setcounter{footnote}{0}

\section{Introduction}

The starting point for classical Lagrangian mechanics is a function $ L \colon T Q \rightarrow \mathbb{R} $, called the \emph{Lagrangian}, where $ T Q $ is the tangent bundle of a smooth conf\/iguration manifold~$Q$.

Yet, tangent bundles are hardly the only spaces on which one may wish to study Lagrangian mechanics. When $L$ is invariant with respect to certain symmetries, it is useful to perform \emph{Lagrangian reduction}: quotienting out the symmetries and thereby passing to a~smaller space than $ T Q $. For example, if a Lie group $G$ acts freely and properly on $Q$, then $ Q \rightarrow Q / G $ is a~principal f\/iber bundle; if $L$ is invariant with respect to the $G$-action, then one can def\/ine a~\emph{reduced Lagrangian} on the quotient $ T Q / G $ (cf.\ Marsden and Scheurle~\cite{MaSc1993b}, Cendra et al.~\cite{CeMaRa2001}). In particular, when $ Q = G $, the reduced Lagrangian is def\/ined on $ T G / G \cong \mathfrak{g} $, the Lie algebra of~$G$, and the reduction procedure is called \emph{Euler--Poincar\'e reduction} (cf.\ Marsden and Ratiu~\cite[Chapter~13]{MaRa1999}).

Unlike $ T Q $, the reduced spaces $ T Q / G $ and $ \mathfrak{g} $ are not tangent bundles~-- but all three are examples of \emph{Lie algebroids}. Beginning with a seminal paper of Weinstein~\cite{Weinstein1996}, and with particularly important follow-up work by Mart\'{\i}nez~\cite{Martinez2001,Martinez2007,Martinez2008}, this has driven the development of a more general theory of Lagrangian mechanics on Lie
algebroids. In this more general framework, reduction is associated with \emph{Lie algebroid morphisms}, of which the quotient map $ T Q \rightarrow T Q / G $ is a particular example. Since Lie algebroids form a category, the composition of two morphisms is again a morphism. As an important consequence, it is almost trivial to perform so-called \emph{reduction by stages}~-- applying a sequence of morphisms one at a time rather than all at once~-- whereas, without this framework, reduction by stages is considerably more dif\/f\/icult (Cendra et al.~\cite{CeMaRa2001}, Marsden et al.~\cite{MaMiOrPeRa2007}).

In this paper, we generalize the foregoing theory in a new direction, based on the observation that reduction from $ T Q $ to $ T Q / G $ is a~special case of a much more general construction, involving \emph{Lie groupoid} (rather than group) actions on \emph{fibered manifolds} (rather than ordinary manifolds). This includes not only Lagrangian reduction, but also the related theory of \emph{Routh reduction}, which we show is naturally posed in the language of f\/ibered manifolds. In the special case of a manifold trivially f\/ibered over a single point, i.e., an ordinary manifold, this reduces to the previously-studied cases. Along the way, we also develop some new results on Lagrangian mechanics on Lie algebroids~-- most notably a~new, coordinate-free formulation of the equations of motion, incorporating the notion of a Lie algebroid connection due to Crainic and Fernandes~\cite{CrFe2003}~-- and extend this theory to the Hamilton--Pontryagin principle of Yoshimura and Marsden~\cite{YoMa2006b}.

The paper is organized as follows:
\begin{itemize}\itemsep=0pt
\item In Section~\ref{sec:EL}, we begin by brief\/ly reviewing the classical formulation of Lagrangian mechanics on manifolds. We then def\/ine f\/ibered manifolds, together with appropriate spaces of vertical tangent vectors and paths, and show how Lagrangian mechanics may be generalized to this setting. As an application, we show that Routh reduction is naturally posed in the language of f\/ibered manifolds, where the classical Routhian is understood as a Lagrangian on an appropriate vertical bundle.

\item In Section~\ref{sec:algebroids}, we discuss Lagrangian mechanics on Lie algebroids. We call the associated equations of motion the \emph{Euler--Lagrange--Poincar\'e equations}, since they simultaneously generalize the Euler--Lagrange equations on $ T Q $, Euler--Poincar\'e equations on $ \mathfrak{g} $, and Lagrange--Poincar\'e equations on $ T Q / G $. We derive a new, coordinate-free formulation of these equations, which we show agrees with the local-coordinates expression previously obtained by Mart\'{\i}nez~\cite{Martinez2001}. Finally, we show that, since the vertical bundle of a f\/ibered manifold is a Lie algebroid, the theory of Section~\ref{sec:EL} can be interpreted in this light.

\item In Section~\ref{sec:reduction}, we employ the Lie algebroid toolkit of Section~\ref{sec:algebroids} to study Lagrangian reduction on f\/ibered
 manifolds by Lie groupoid actions, which we call \emph{Euler--Lagrange--Poincar\'e reduction}. In the special case where a Lie groupoid acts on itself by multiplication, we recover the theory of Lagrangian mechanics on its associated Lie algebroid.

\item Finally, in Section~\ref{sec:HP}, we generalize the \emph{Hamilton--Pontryagin} variational principle of Yoshimura and Marsden~\cite{YoMa2006b}, together with the associated reduction theory \cite{YoMa2007}, to f\/ibered manifolds with Lie groupoid symmetries.
\end{itemize}

\section{Lagrangian mechanics on f\/ibered manifolds}\label{sec:EL}

\subsection{Brief review of Lagrangian mechanics}\label{sec:review}

Let $Q$ be a smooth conf\/iguration manifold and $ L \colon T Q \rightarrow \mathbb{R} $ be a smooth function, called the Lagrangian, on its tangent bundle. There are three ways in which one can use $L$ to induce dynamics on $Q$.

The f\/irst, which we call the symplectic approach, begins by introducing the \emph{Legendre transform} (or f\/iber derivative) of~$L$, which is the bundle map $ \mathbb{F} L \colon T Q \rightarrow T ^\ast Q $ def\/ined f\/iberwise by
\begin{gather*}
 \bigl\langle \mathbb{F} _q L (v) , w \bigr\rangle = \frac{\mathrm{d}}{\mathrm{d}t} L ( v + t w ) \Bigr\rvert _{ t = 0 }, \qquad v, w \in T _q Q .
\end{gather*}
This is used to pull back the canonical symplectic form $ \omega \in \Omega ^2 ( T ^\ast Q ) $ to the \emph{Lagrangian $2$-form} $ \omega _L = ( \mathbb{F} L ) ^\ast \omega \in \Omega ^2 ( T Q ) $. The Lagrangian is said to be \emph{regular} if $ \mathbb{F} L $ is a~local bundle isomorphism; in this case, $ \omega _L $ is nondegenerate, so $ ( T Q , \omega _L ) $ is a symplectic manifold. The \emph{energy function} $ E _L \colon T Q \rightarrow \mathbb{R} $ associated to $L$ is
\begin{gather*}
 E _L (v) = \bigl\langle \mathbb{F} L (v) , v \bigr\rangle - L (v) ,
\end{gather*}
and the \emph{Lagrangian vector field} $ X _L \in \mathfrak{X} ( T Q ) $ is the vector f\/ield satisfying
\begin{gather*}
 i _{ X _L } \omega _L = \mathrm{d} E _L ,
\end{gather*}
where $ i _{ X _L } \omega _L = \omega _L ( X _L , \cdot ) $ is the interior product of $ X _L $ with $ \omega _L $. That is, $ X _L $ is the Hamiltonian vector f\/ield of $ E _L $ on the symplectic manifold $ ( T Q, \omega _L ) $. Finally, a $ C ^2 $ path $ q \colon I \rightarrow Q $ is called a~\emph{base integral curve} of $ X _L $ if its tangent prolongation $ ( q, \dot{q} ) \colon I \rightarrow T Q $ is an integral curve of~$ X _L $. (Here and henceforth, $I$ denotes the closed unit interval $ [0,1] $, but there is no loss of generality over any other closed interval $[a,b]$.)

The second, which we call the variational approach, begins with the \emph{action functional} $ S \colon \mathcal{P} (Q) \rightarrow \mathbb{R} $, def\/ined by the integral
\begin{gather*}
 S (q) = \int _0 ^1 L \bigl( q (t) , \dot{q} (t) \bigr) \,\mathrm{d}t ,
\end{gather*}
where $ \mathcal{P} (Q) $ denotes the Banach manifold of $ C ^2 $ paths $ q \colon I \rightarrow Q $. A path $q \in \mathcal{P} (Q) $ satisf\/ies \emph{Hamilton's variational principle} if it is a~critical point of $S$ restricted to paths with f\/ixed endpoints~$ q (0) $ and~$ q (1) $, i.e., if $ \mathrm{d} S (\delta q) = 0 $ for all variations $ \delta q \in T _q \mathcal{P} (Q) $ with $ \delta q (0) = 0 $ and $ \delta q (1) = 0 $.

The third and f\/inal approach considers the system of dif\/ferential equations that a solution to Hamilton's variational principle must satisfy. In local coordinates, assuming $ \delta q (0) = 0 $ and $ \delta q (1) = 0 $,
\begin{gather*}
 \mathrm{d} S ( \delta q ) = \int _0 ^1 \left( \frac{ \partial L }{ \partial q ^i } ( q, \dot{q} ) \delta q ^i + \frac{ \partial L }{ \partial \dot{q} ^i } ( q , \dot{q} ) \delta \dot{q} ^i \right) \mathrm{d}t
 = \int _0 ^1 \left( \frac{ \partial L }{ \partial q ^i } ( q, \dot{q} ) - \frac{\mathrm{d}}{\mathrm{d}t} \frac{ \partial L }{ \partial \dot{q}
 ^i } ( q, \dot{q} ) \right) \delta q ^i \,\mathrm{d}t .
\end{gather*}
(Here, we use the Einstein index convention, where there is an implicit sum on repeated indices.) Hence, this vanishes for all $ \delta q $ if and only if $q$ satisf\/ies the system of ordinary dif\/ferential equations
\begin{gather*}
 \frac{ \partial L }{ \partial q ^i } ( q, \dot{q} ) - \frac{\mathrm{d}}{\mathrm{d}t} \frac{ \partial L }{ \partial \dot{q} ^i } ( q, \dot{q} ) = 0,
\end{gather*}
which are called the \emph{Euler--Lagrange equations}.

The equivalence of these three approaches for regular Lagrangians~-- and of the latter two for arbitrary Lagrangians~-- is a~standard result in geometric mechanics. We state it now as a~theorem for later reference.

\begin{Theorem}\label{thm:lagrangian} If $ L \colon T Q \rightarrow \mathbb{R} $ is a regular Lagrangian and $ q \in \mathcal{P} (Q) $, then the following are equivalent:
\begin{enumerate}[label={\rm (\roman*)}]\itemsep=0pt
\item $q$ is a base integral curve of the Lagrangian vector field $ X _L \in \mathfrak{X} ( T Q ) $.

\item $q$ satisfies Hamilton's variational principle.

\item $q$ satisfies the Euler--Lagrange equations.
\end{enumerate}
 If regularity is dropped, then $(ii)\Leftrightarrow (iii)$ still holds.
\end{Theorem}

\begin{proof}
 See, e.g., Marsden and Ratiu~\cite[Theorem~8.1.3]{MaRa1999}.
\end{proof}

\subsection{Fibered manifolds}\label{sec:fibered}

We begin by giving the def\/inition of a f\/ibered manifold, along with its vertical and covertical bundles. These bundles generalize the tangent and cotangent bundles of an ordinary manifold, and they will play analogous roles in f\/ibered Lagrangian mechanics.

\begin{Definition}
A \emph{fibered manifold} $ Q \rightarrow M $ consists of a pair of smooth manifolds $ Q $, $M$, together with a surjective submersion $ \mu \colon Q \rightarrow M $.
\end{Definition}

\begin{Definition}
The \emph{vertical bundle} of $ Q \rightarrow M $ is $ V Q = \ker \mu _\ast $, where $ \mu _\ast \colon T Q \rightarrow T M $ is the pushforward (or tangent map) of $\mu$. The dual of $ V Q $ is denoted $ V ^\ast Q $, which we call the \emph{covertical bundle}.
\end{Definition}

\begin{Remark} \label{rmk:disjoint}
Since $\mu$ is a submersion, the f\/iber $ Q _x = \mu ^{-1} \bigl( \{ x \} \bigr) $ is a submanifold of $Q$ for each $ x \in M $. Therefore,
 \begin{gather*}
 V _q Q = T _q Q _{\mu (q) } , \qquad V Q = \bigsqcup _{ q \in Q } V_q Q = \bigsqcup _{ x \in M } T Q _x .
 \end{gather*}
In other words, $ V Q $ consists of vectors tangent to the f\/ibers, and hence is an integrable subbundle of $ T Q $. Similarly,
\begin{gather*}
 V ^\ast_q Q = T ^\ast _q Q _{ \mu (q) } , \qquad V ^\ast Q = \bigsqcup _{ q \in Q } V _q ^\ast Q = \bigsqcup _{ x \in M } T ^\ast Q _x ,
 \end{gather*}
so the covertical bundle consists of covectors to the individual f\/ibers.
\end{Remark}

\begin{Example} An ordinary smooth manifold $Q$ can be identif\/ied with the f\/ibered manifold $ Q \rightarrow \bullet $, where $ \bullet $ denotes the space with a single point. Because $ \mu _\ast $ is trivial, it follows that $ V Q = T Q $ and $ V ^\ast Q = T ^\ast Q $.
\end{Example}

\begin{Definition} The space of \emph{vertical vector fields} on $Q$ is $ \mathfrak{X} _V (Q) = \Gamma ( V Q ) $. The space of \emph{vertical $k$-forms} on $Q$ is $ \Omega _V ^k ( Q ) = \Gamma ( \bigwedge ^k V ^\ast Q ) $.
\end{Definition}

Since $ V Q $ is integrable, it follows that $ \mathfrak{X} _V ( Q ) $ is closed under the Jacobi--Lie bracket $ [ \cdot , \cdot ] $, i.e., $\mathfrak{X} _V (Q) $ is a Lie subalgebra of $ \mathfrak{X} (Q) $. Therefore, the following vertical exterior derivative operator on $ \Omega _V ^\bullet (Q) $ is well-def\/ined.

\begin{Definition} Given $ u \in \Omega ^k _V (Q) $, the \emph{vertical exterior derivative} $ \mathrm{d} _V u \in \Omega ^{ k + 1 } _V (Q) $
 is given by
\begin{gather*}
\mathrm{d} _V u ( X _0 , \ldots, X _k ) = \sum _{ i = 0 } ^k ( - 1 ) ^i X _i \bigl[ u ( X _0 , \ldots, \widehat{ X } _i , \ldots, X _k ) \bigr] \\
\hphantom{\mathrm{d} _V u ( X _0 , \ldots, X _k ) =}{} + \sum _{ 0 \leq i < j \leq k } ( - 1 ) ^{ i + j } \big( [X _i , X _j ], X _0 , \ldots, \widehat{ X } _i , \ldots, \widehat{ X } _j , \ldots, X _k \big) ,
 \end{gather*}
where $ X _0 , \ldots, X _k \in \mathfrak{X} _V (Q) $ are arbitrary vertical vector f\/ields, and where a hat over an argument indicates its omission.
\end{Definition}

\begin{Remark} From the characterization of $ V Q $ in Remark~\ref{rmk:disjoint}, it follows that $ X \in \mathfrak{X} _V (Q) $ restricts to an ordinary vector f\/ield $ X _x \in \mathfrak{X} ( Q _x ) $ on each f\/iber $Q _x$. Likewise, $ u \in \Omega ^k _V ( Q ) $ restricts to an ordinary $k$-form $ u _x \in \Omega ^k ( Q _x ) $ on each f\/iber $ Q _x $. Moreover, by the integrability of $V Q$, for any $ X, Y \in \mathfrak{X} _V (Q) $ and $ x \in M $ we have $ [ X, Y ] _x = [ X _x , Y _x ] \in \mathfrak{X} ( Q _x ) $. Hence, the vertical exterior derivative $ \mathrm{d} _V $ coincides with the ordinary exterior derivative $ \mathrm{d} _x \colon \Omega ^k ( Q _x ) \rightarrow \Omega ^{ k + 1 } ( Q _x ) $ on the f\/iber $ Q _x $.
\end{Remark}

Note that $ V Q $ and $ V ^\ast Q $ are also, themselves, f\/ibered manifolds over $M$. Specif\/ically, if $ \tau \colon V Q \rightarrow Q $ and $ \pi \colon V ^\ast Q \rightarrow Q $ are the bundle projections, then we have surjective submersions $ \mu \circ \tau \colon V Q \rightarrow M $ and
$ \mu \circ \pi \colon V ^\ast Q \rightarrow M $; the f\/ibers are given by $ (V Q) _x = T Q _x $ and $ (V ^\ast Q) _x = T ^\ast Q _x $. Now,
just as there is a tautological $1$-form and canonical symplectic $2$-form on $ T ^\ast Q $, there are corresponding vertical forms on $ V ^\ast Q $, constructed as follows.

\begin{Definition} The \emph{tautological vertical $1$-form} $ \theta \in \Omega ^1 _V ( V ^\ast Q ) $ is def\/ined by the condition $ \theta ( v ) = \langle p , \pi _\ast v \rangle $ for $ v \in V _p V ^\ast Q $. The \emph{canonical vertical $2$-form} is def\/ined by $ \omega = - \mathrm{d} _V \theta \in \Omega ^2 _V ( V ^\ast Q ) $.
\end{Definition}

\begin{Remark} Restricted to any f\/iber $ ( V ^\ast Q ) _x = T ^\ast Q _x $, it follows from the preceding remarks that $\theta$ and $\omega$ agree with the ordinary tautological $1$-form $\theta _x \in \Omega ^1 ( T ^\ast Q _x ) $ and canonical symplectic $2$-form $\omega _x \in \Omega ^2 ( T ^\ast Q _x ) $, respectively, on the cotangent bundle of the f\/iber. In particular, this implies that $ \omega $ is closed (with respect to $ \mathrm{d} _V $) and nondegenerate, since $ \omega _x $ is closed and nondegenerate for each $ x \in M $.
\end{Remark}

\subsection{Lagrangian mechanics on f\/ibered manifolds}

In this section, we show that the three approaches to Lagrangian mechanics of Section~\ref{sec:review} may be generalized to f\/ibered manifolds, with a corresponding generalization of Theorem~\ref{thm:lagrangian}. Let the Lagrangian be a smooth function $ L \colon V Q \rightarrow \mathbb{R} $.

\begin{Definition} The \emph{Legendre transform} (or f\/iber derivative) of $L$ is the bundle map $ \mathbb{F} L \colon V Q \rightarrow V ^\ast Q $,
 def\/ined for each $ q \in Q $ by
 \begin{gather*}
 \bigl\langle \mathbb{F} _q L (v) , w \bigr\rangle = \frac{\mathrm{d}}{\mathrm{d}t} L ( v + t w ) \Bigr\rvert _{ t = 0 }, \qquad v, w \in V _q Q .
 \end{gather*}
We say that $L$ is \emph{regular} if $ \mathbb{F} L $ is a local bundle isomorphism.
\end{Definition}

\begin{Remark} Since $ ( V Q ) _x = T Q _x $, we can def\/ine a f\/iber-restricted Lagrangian $ L _x \colon T Q _x \rightarrow \mathbb{R} $, whose ordinary Legendre transform $ \mathbb{F} L _x \colon T Q _x \rightarrow T ^\ast Q _x $ coincides with the restriction $ \mathbb{F} L \rvert _{ (V Q) _x } $. It is therefore useful to think of $L$ as a smoothly varying family of ordinary Lagrangians $ L _x $, parametrized by $ x \in M $.
\end{Remark}

Now, $ \mathbb{F} L $ maps f\/ibers to f\/ibers (i.e., it is a morphism of f\/ibered manifolds over $M$), so its pushforward maps vertical vectors
to vertical vectors, and we may write $( \mathbb{F} L ) _\ast \colon V V Q \rightarrow V V ^\ast Q $. This also gives a well-def\/ined pullback of vertical forms $ ( \mathbb{F} L ) ^\ast \colon \Omega ^k _V ( V ^\ast Q ) \rightarrow \Omega ^k _V ( V Q ) $, which leads to the following vertical versions of the Lagrangian $2$-form and Lagrangian vector f\/ield.

\begin{Definition} The \emph{Lagrangian vertical $2$-form} is $ \omega _L = ( \mathbb{F} L ) ^\ast \omega \in \Omega ^2 _V ( V Q )
 $. The \emph{Lagrangian vertical vector field} $ X _L \in \mathfrak{X} _V ( V Q ) $ is the vertical vector f\/ield satisfying
 \begin{gather*}
 i _{ X _L } \omega _L = \mathrm{d} _V E _L ,
 \end{gather*}
where the energy function $ E _L \colon V Q \rightarrow \mathbb{R} $ is given by $E _L (v) = \bigl\langle \mathbb{F} L (v) , v \bigr\rangle - L (v)$.
\end{Definition}

\begin{Remark} Restricting to the f\/iber over $ x \in M $, we have
\begin{gather*}
 ( \omega _L ) _x = ( \mathbb{F} L _x ) ^\ast \omega _x = \omega _{L _x} \in \Omega ^2 ( T Q _x ) ,
\end{gather*}
i.e., the ordinary Lagrangian $2$-form for $ L _x $ on $ T Q _x $, and moreover
\begin{gather*}
 ( E _L ) _x (v) = \bigl\langle \mathbb{F} L _x (v) , v \bigr\rangle- L _x (v) = E _{ L _x } (v) , \qquad v \in T Q _x ,
\end{gather*}
so $ E _L $ restricts to $ E _{ L _x } $. Combining these, it follows that
\begin{gather*}
 i _{ (X _L) _x } \omega _{ L _x } = \mathrm{d} _x E _{ L _x } ,
\end{gather*}
so we conclude that $ ( X _L ) _x = X _{ L _x } $, i.e., $ X _L $ coincides with the ordinary Lagrangian vector f\/ield on each f\/iber.
\end{Remark}

Next, for the variational approach, we begin by def\/ining an appropriate space of vertical paths on which the action functional will be def\/ined, as well as an appropriate space of variations of these paths.

\begin{Definition} The space of \emph{$ C ^2 $ vertical paths}, denoted by $ \mathcal{P} _V (Q) \subset \mathcal{P} (Q) $, consists of $ q \in \mathcal{P} (Q) $ whose tangent prolongation satisf\/ies $ \bigl( q (t) , \dot{q} (t) \bigr) \in V Q $ for all $ t \in I $. The \emph{action functional} $ S \colon \mathcal{P} _V (Q) \rightarrow \mathbb{R} $ is then
 \begin{gather*}
 S (q) = \int _0 ^1 L \bigl( q (t) , \dot{q} (t) \bigr) \,\mathrm{d}t ,
 \end{gather*}
which is well-def\/ined for $ L \colon V Q \rightarrow \mathbb{R} $ since $ \bigl( q (t) , \dot{q} (t) \bigr) \in V Q $.
\end{Definition}

\begin{Remark} For $ q \in \mathcal{P} _V (Q) $, the condition $ \bigl( q (t) , \dot{q} (t) \bigr) \in V Q $ implies
 \begin{gather*}
 \frac{\mathrm{d}}{\mathrm{d}t} \mu \bigl( q (t) \bigr) = \mu _\ast \bigl( q (t) , \dot{q} (t) \bigr) = 0 .
 \end{gather*}
Hence, $ \mu \bigl( q (t) \bigr) $ is constant in $t$, so $q$ lies in a single f\/iber $ Q _x $, i.e., $ q \in \mathcal{P} ( Q _x ) $ for some $ x \in M $. It follows that $ S (q) = S _x ( q ) $, where $ S _x \colon \mathcal{P} ( Q _x ) \rightarrow \mathbb{R} $ is the ordinary action associated to the f\/iber-restricted Lagrangian $ L _x $. Moreover, since $ \mu \bigl( q (t) \bigr) $ is constant in~$t$, there is an associated f\/ibered (Banach) manifold structure $ \mathcal{P} _V (Q) \rightarrow M $, with $ \mathcal{P} _V (Q) _x = \mathcal{P} ( Q _x ) $.
\end{Remark}

\begin{Definition} \label{def:verticalHamilton}
An element $ \delta q \in V _q \mathcal{P} _V (Q) $ is called a~\emph{vertical variation} of $ q \in \mathcal{P} _V (Q) $. The path~$q$ satisf\/ies \emph{Hamilton's variational principle for vertical paths} if $q$ is a critical point of $S$ relative to paths with f\/ixed endpoints, i.e., if $ \mathrm{d} S ( \delta q ) = 0 $ for all vertical variations $ \delta q $ with $ \delta q (0) = 0 $ and $ \delta q (1) = 0 $.
\end{Definition}

\begin{Remark} Since $ \mathcal{P} _V (Q) _x = \mathcal{P} ( Q _x ) $ and $ V _q \mathcal{P} _V (Q) = T _q \mathcal{P} ( Q _x ) $, this is immediately equivalent to $q \in \mathcal{P} _V (Q) _x $ satisfying the ordinary form of Hamilton's variational principle for the f\/iber-restricted Lagrangian $ L _x $.
\end{Remark}

Having def\/ined vertical versions of the symplectic and variational approaches to Lagrangian mechanics, we f\/inally derive the corresponding Euler--Lagrange equations. Suppose that $ q = ( x ^\sigma , y ^i ) $ are f\/iber-adapted local coordinates for~$Q$. Since vertical variations satisfy $ \delta x ^\sigma = 0 $, by def\/inition, arbitrary f\/ixed-endpoint variations of the action functional are given by
\begin{gather*}
 \mathrm{d} S ( \delta q )= \int _0 ^1 \left( \frac{ \partial L }{ \partial y ^i } ( q, \dot{q} ) \delta y ^i + \frac{ \partial L }{ \partial \dot{y} ^i } ( q, \dot{q} ) \delta \dot{y} ^i \right) \mathrm{d}t = \int _0 ^1 \left( \frac{ \partial L }{ \partial y ^i } ( q, \dot{q} ) - \frac{\mathrm{d}}{\mathrm{d}t} \frac{ \partial L }{ \partial \dot{y} ^i } ( q, \dot{q} ) \right) \delta y ^i
 \,\mathrm{d}t .
\end{gather*}
Therefore, a critical vertical path must have the integrand above vanish, in addition to the vertical path condition. This motivates the following def\/inition.

\begin{Definition}
 In f\/iber-adapted local coordinates $ q = ( x ^\sigma , y ^i ) $ on $Q \rightarrow M $, the \emph{vertical Euler--Lagrange equations} for $ L \colon V Q \rightarrow \mathbb{R} $ are
 \begin{gather} \label{eqn:verticalEL}
 \dot{x} ^\sigma = 0 , \qquad \frac{ \partial L }{ \partial y ^i } ( q, \dot{q} ) - \frac{\mathrm{d}}{\mathrm{d}t} \frac{ \partial L }{ \partial \dot{y}
 ^i } ( q, \dot{q} ) = 0 .
\end{gather}
\end{Definition}

\begin{Remark} Since $ q = ( x ^\sigma , y ^i ) $ are f\/iber-adapted local coordinates, $ y ^i $ gives local coordinates for the f\/iber $ Q _x $, and we may write $ L ( q, \dot{q} ) = L _x ( y , \dot{y} ) $. (Note that $L$ is def\/ined only on vertical tangent vectors, so $ \dot{x} $ is not required.) Therefore, the vertical Euler--Lagrange equations are equivalent to the ordinary Euler--Lagrange equations,
 \begin{gather*}
 \frac{ \partial L _x }{ \partial y ^i } ( y , \dot{y} ) - \frac{\mathrm{d}}{\mathrm{d}t} \frac{ \partial L _x }{ \partial \dot{y} ^i } ( y , \dot{y} ) = 0 ,
 \end{gather*}
 for the f\/iber-restricted Lagrangian $ L _x $.
\end{Remark}

We are now prepared to state the generalization of Theorem~\ref{thm:lagrangian} to Lagrangian mechanics on f\/ibered manifolds.

\begin{Theorem} \label{thm:fiberedLagrangian}
If $ L \colon V Q \rightarrow \mathbb{R} $ is a regular Lagrangian on a fibered manifold $ \mu \colon Q \rightarrow M $, and $ q \in \mathcal{P} _V (Q) $ is a vertical $ C ^2 $ path over $ x \in M $, then the following are equivalent:
\begin{enumerate}[label={\rm (}\roman*{\rm )}]\itemsep=0pt
\item $q$ is a base integral curve of the Lagrangian vector field $ X _L \in \mathfrak{X} _V ( V Q ) $.
\item $q$ satisfies Hamilton's variational principle for vertical paths.
\item $q$ satisfies the vertical Euler--Lagrange equations.
\end{enumerate}
\begin{enumerate}[label={\rm (}\roman*$\,{}^\prime${\rm )}]\itemsep=0pt
\item $q$ is a base integral curve of the fiber-restricted Lagrangian vector field $ X _{ L _x } \in \mathfrak{X} ( T Q _x ) $.
\item $q$ satisfies Hamilton's variational principle with respect to the fiber-restricted Lagrangian $ L _x $.
\item $q$ satisfies the Euler--Lagrange equations with respect to the fiber-restricted Lagrangian $ L _x $.
\end{enumerate}
If regularity is dropped, then $ ( ii ) \Leftrightarrow ( iii ) \Leftrightarrow ( ii ^\prime )\Leftrightarrow ( iii ^\prime ) $ still holds.
\end{Theorem}

\begin{proof}
We have seen, in the foregoing discussion, that $(i)$ $ \Leftrightarrow $ $(i ^\prime)$ for regular Lagrangians, while $(ii)$ $ \Leftrightarrow $ $(ii ^\prime)$ and $(iii)$ $ \Leftrightarrow $ $(iii ^\prime)$ hold in general. Hence, it suf\/f\/ices to show $(i ^\prime)$ $ \Leftrightarrow $ $(ii ^\prime)$ $ \Leftrightarrow $ $(iii ^\prime)$ for the regular case and $(ii ^\prime)$ $ \Leftrightarrow $ $(iii ^\prime)$ for the general case~-- but this is simply Theorem~\ref{thm:lagrangian} applied to~$L _x$.
\end{proof}

\subsection{Application: classical Routh reduction as f\/ibered mechanics}

The technique known as \emph{Routh reduction} traces its origins as far back as the 1860 treatise of Routh~\cite{Routh1860}. Modern geometric accounts have been given by Arnold et al.~\cite{ArKoNe1988}, Marsden and Scheurle~\cite{MaSc1993}, and Marsden et al.~\cite{MaRaSc2000}, with the latter two works developing a~more general theory of \emph{nonabelian Routh reduction}.

The essence of Routh reduction, as we will show, is that it passes from a Lagrangian on an ordinary manifold to an equivalent Lagrangian, known as the \emph{Routhian}, on a f\/ibered manifold. Since the resulting dynamics are conf\/ined to the vertical components (i.e., restricted to individual f\/ibers), this reduces the size of the original system by eliminating the horizontal components.

Consider a conf\/iguration manifold of the form $ \mathbb{T} ^n \times \mathcal{S} $, where $ \mathbb{T} ^n $ denotes the $n$-torus and $\mathcal{S}$ is a manifold called the \emph{shape space}. Let $ \theta ^\sigma $ and $ y ^i $ be local coordinates for $ \mathbb{T} ^n $ and $\mathcal{S}$, respectively, and suppose the Lagrangian $ L \colon T ( \mathbb{T} ^n \times S ) \rightarrow \mathbb{R} $ is \emph{cyclic} in the variables $ \theta ^\sigma $, i.e.,
$L = L ( \dot{\theta} , y , \dot{y} ) $ depends only on $ \dot{ \theta } $ but not on $ \theta $ itself. Then the $ \theta ^\sigma $ components of the Euler--Lagrange equations imply that
\begin{gather*}
 \frac{\mathrm{d}}{\mathrm{d}t} \frac{ \partial L }{ \partial \dot{\theta} ^\sigma } = \frac{ \partial L }{ \partial \theta
 ^\sigma } = 0,
\end{gather*}
so $ x _\sigma = \partial L / \partial \dot{\theta} ^\sigma $ is constant in $t$.

Now, def\/ine the f\/ibered manifold $ Q = \mathbb{R}^n \times \mathcal{S} $ with $ M = \mathbb{R}^n $, where $ \mu \colon Q \rightarrow M $ is simply the projection onto the $\mathbb{R}^n$ component, so that $ V Q = \mathbb{R}^n \times T \mathcal{S} $. The \emph{classical Routhian} $ R \colon V Q \rightarrow \mathbb{R} $ is
\begin{gather} \label{eqn:classicalRouthian}
 R ( x, y , \dot{y} ) = \bigl[ L ( \dot{\theta} , y, \dot{y} ) - x _\sigma \dot{\theta} ^\sigma \bigr] _{ x _\sigma = \partial L / \partial \dot{\theta} ^\sigma } ,
\end{gather}
where each $ \dot{\theta} ^\sigma $ is determined implicitly by the constraint $ x _\sigma = \partial L / \partial \dot{\theta} ^\sigma $. Considering $R$ as a~Lagrangian in the sense of the previous section, the vertical Euler--Lagrange equations consist of the vertical path condition,
\begin{gather*}
 0 = \dot{x} _\sigma = \frac{\mathrm{d}}{\mathrm{d}t} \frac{ \partial L }{ \partial \dot{\theta} ^\sigma },
\end{gather*}
and
\begin{gather*}
 0 = \frac{ \partial R }{ \partial y ^i } ( x, y , \dot{y} ) - \frac{\mathrm{d}}{\mathrm{d}t} \frac{ \partial R }{ \partial
 \dot{y} ^i } ( x, y , \dot{y} ) \\
\hphantom{0}{}
= \left( \frac{ \partial L }{ \partial y ^i } + \frac{ \partial L }{ \partial \dot{\theta} ^\sigma } \frac{ \partial
 \dot{\theta} ^\sigma }{ \partial y ^i } - x _\sigma
 \frac{ \partial \dot{\theta} ^\sigma }{ \partial y ^i } \right) -\frac{\mathrm{d}}{\mathrm{d}t} \left( \frac{ \partial L
 }{ \partial \dot{y} ^i } + \frac{ \partial L }{ \partial \dot{\theta} ^\sigma } \frac{ \partial \dot{\theta} ^\sigma }{ \partial \dot{y} ^i } -
 x _\sigma \frac{ \partial \dot{\theta} }{ \partial \dot{y} ^i } \right) \\
\hphantom{0}{} = \frac{ \partial L }{ \partial y ^i } - \frac{\mathrm{d}}{\mathrm{d}t} \frac{ \partial L
 }{ \partial \dot{y} ^i },
\end{gather*}
where the last step uses $ x _\sigma = \partial L / \partial \dot{\theta} ^\sigma $ to eliminate the last two terms from each parenthetical expression.

Thus, the ordinary Euler--Lagrange equations for $L$ are precisely equivalent to the vertical Euler--Lagrange equations for $R$. This reduces the dynamics from $\mathbb{T} ^n \times \mathcal{S} $ to those on the individual f\/ibers $ Q _x \cong \mathcal{S} $, thereby eliminating the cyclic variables $\theta \in \mathbb{T} ^n $. We now summarize this result as a theorem.

\begin{Theorem} Suppose $ L \colon T ( \mathbb{T} ^n \times \mathcal{S} ) \rightarrow \mathbb{R} $ is an ordinary Lagrangian that is cyclic in the $ \mathbb{T} ^n $ components, and let the classical Routhian $ R \colon V ( \mathbb{R}^n \times \mathcal{S} ) \rightarrow \mathbb{R} $ be the fibered Lagrangian defined in~\eqref{eqn:classicalRouthian}. Then $ ( \theta , y ) \in \mathcal{P} ( \mathbb{T} ^n \times S ) $ is a~solution path for $L$ if and only if $ ( x, y ) \in \mathcal{P} _V ( \mathbb{R} ^n \times \mathcal{S} ) $ is a vertical solution path for $R$.
\end{Theorem}

\begin{proof}
This follows from Theorem~\ref{thm:fiberedLagrangian}, together with the foregoing calculations.
\end{proof}

\section{Lagrangian mechanics on Lie algebroids}\label{sec:algebroids}

In this section, we lay the groundwork for reduction theory on f\/ibered manifolds, which will be discussed in Section~\ref{sec:reduction}. In ordinary Lagrangian reduction, we pass from the tangent bundle $ T Q $ to the quotient $ T Q / G $, which is generally not a tangent bundle. Likewise, in Section~\ref{sec:reduction}, we will pass from vertical bundles to quotients that are generally not vertical bundles. However, $ T Q $ and $ T Q / G $~-- as well as their vertical analogs, as we will show~-- are all examples of more general objects called \emph{Lie algebroids}, on which Lagrangian mechanics can be studied. The study of Lagrangian mechanics on Lie algebroids was largely pioneered by Weinstein~\cite{Weinstein1996}, and important follow-up work was done by Mart\'{\i}nez~\cite{Martinez2001,Martinez2007,Martinez2008} and several others in more recent years; see also Cort\'es et al.~\cite{CoLeMaMaMa2006}, Cort\'es and Mart\'{\i}nez~\cite{CoMa2004}, Grabowska and Grabowski~\cite{GrGr2008}, Grabowska et al.~\cite{GrUrGr2006}, Iglesias et al.~\cite{IgMaMaMa2008, IgMaMaSo2008}.

In addition to recalling some of the key results (particularly of Weinstein~\cite{Weinstein1996} and Mart\'{\i}nez~\cite{Martinez2001}) that we will need for the subsequent reduction theory, we also develop a new, coordinate-free formulation of the equations of motion, which we call the \emph{Euler--Lagrange--Poincar\'e equations} (since they simultaneously generalize the Euler--Lagrange, Euler--Poincar\'e, and Lagrange--Poincar\'e equations). This new formulation is based on the work of Crainic and Fernandes~\cite{CrFe2003}, particularly the notion of a Lie algebroid connection and its use in describing variations of paths.

\subsection[Lie algebroids and $A$-paths]{Lie algebroids and $\boldsymbol{A}$-paths}

We begin by recalling the def\/inition of a Lie algebroid $A$ and an appropriate class of paths in~$A$, called $A$-paths. This review will necessarily be very brief, but for more information on Lie algebroids, we refer the reader to the comprehensive work by Mackenzie~\cite{Mackenzie2005}.

\begin{Definition} A \emph{Lie algebroid} is a real vector bundle $ \tau \colon A \rightarrow Q $ equipped with a Lie bracket $ [ \cdot , \cdot ] \colon \Gamma (A) \times \Gamma (A) \rightarrow \Gamma (A) $ on its space of sections and a bundle map $ \rho \colon A \rightarrow T Q $, called the \emph{anchor map}, satisfying the following Leibniz rule-like compatibility condition:
 \begin{gather*}
 [ X, f Y ] = f [X, Y] + \rho (X) [f] Y, \qquad \text{for all $ X, Y \in \Gamma (A) $, $ f \in C ^\infty (Q) $}.
 \end{gather*}
\end{Definition}

\begin{Example} \label{ex:tangentAlgebroid}
The tangent bundle $ T Q $ is a Lie algebroid over $Q$, where $ \tau \colon T Q \rightarrow Q $ is the usual bundle projection, $ [ \cdot, \cdot ] \colon \mathfrak{X} (Q) \times \mathfrak{X} (Q) \rightarrow \mathfrak{X} (Q) $ is the Jacobi--Lie bracket of vector f\/ields, and $ \rho \colon T Q \rightarrow T Q $ is the identity.

Furthermore, any integrable distribution $ \mathcal{D} \subset T Q $ is also a Lie algebroid over $Q$, where $ \tau $, $ [ \cdot , \cdot ] $, and $ \rho $ are just the restrictions to $\mathcal{D}$ of the corresponding maps for $ T Q $. We say that $\mathcal{D}$ is a \emph{Lie subalgebroid} of $ T Q $.

In particular, if $Q \rightarrow M $ is a f\/ibered manifold, then $ V Q \subset T Q $ is integrable and hence a Lie algebroid over~$Q$. (Note that $ V Q $ is generally \emph{not} a Lie algebroid over~$M$, since it may not even be a vector bundle over $M$.)
\end{Example}

\begin{Example} Any Lie algebra $ \mathfrak{g} $ is a Lie algebroid over $ \bullet $ (the space with one point), where the maps $\tau$ and $\rho$ are trivial and $ [ \cdot , \cdot ] $ is the Lie bracket.

More generally, if $ Q \rightarrow Q/G $ is a principal $G$-bundle for some Lie group $G$, then $ T Q / G $ def\/ines an algebroid over $ Q / G $ called the \emph{Atiyah algebroid}. The algebroid $ \mathfrak{g} \rightarrow \bullet $ can be identif\/ied with the special case $ Q = G $, where $G$ is the Lie group integrating $ \mathfrak{g} $ (which exists by Lie's third theorem).
\end{Example}

\begin{Definition} \label{def:a-path}
A path $ a \in \mathcal{P} (A) $ over the base path $ q = \tau \circ a \in \mathcal{P} (Q) $ is called an \emph{$A$-path} if $ \dot{q} (t) = \rho \bigl( a (t) \bigr) $ for all $ t \in I $. The space of $A$-paths is denoted by $ \mathcal{P} _\rho (A) $.
\end{Definition}

\begin{Remark}
Equivalently, $a$ is an $A$-path if and only if $ a \,\mathrm{d}t \colon T I \rightarrow A $ is a morphism of Lie algebroids, where $ T I \rightarrow I $ has the tangent bundle Lie algebroid structure of Example~\ref{ex:tangentAlgebroid}. Hence, $A$-paths can be seen as ``paths in the category of Lie
 algebroids''.
\end{Remark}

\subsection[Connections and variations of $A$-paths]{Connections and variations of $\boldsymbol{A}$-paths}

We now turn to discussing an appropriate class of variations on the space of $A$-paths, $ \mathcal{P} _\rho (A) $. Crainic and Fernandes~\cite[Lemma~4.6]{CrFe2003} show that $ \mathcal{P} _\rho (A) \subset \mathcal{P} (A) $ is a Banach submanifold. However, we do not want to take arbitrary variations $ \delta a \in T _a \mathcal{P} _\rho (A) $, just as we did not want to take arbitrary paths in $ \mathcal{P} (A) $.

To illustrate the reasoning behind this, consider the case of a Lie algebra $ \mathfrak{g} $. Since this is a Lie algebroid over $ \bullet $, where $\tau$ and $\rho$ are trivial, it follows that \emph{every} path $ \xi \in \mathcal{P} ( \mathfrak{g} ) $ is a~$ \mathfrak{g} $-path, i.e., $ \mathcal{P} _\rho ( \mathfrak{g} ) = \mathcal{P} (\mathfrak{g} )$. However, the variational principle for the Euler--Poincar\'e equations on $ \mathfrak{g} $ considers only variations of the form
\begin{gather*}
 \delta \xi = [ \xi, \eta ] + \dot{ \eta } = \operatorname{ad} _\xi \eta + \dot{ \eta } ,
\end{gather*}
where $ \eta \in \mathcal{P} ( \mathfrak{g} ) $ is an arbitrary path vanishing at the endpoints (cf.\ Marsden and Ratiu~\cite[Chapter~13]{MaRa1999}). These constraints on admissible variations are known as \emph{Lin constraints}.

To generalize these constrained variations to an arbitrary Lie algebroid $ A \rightarrow Q $, we f\/irst discuss the notion of a~connection on a~Lie algebroid, of which the adjoint action $ ( \xi, \eta ) \mapsto \operatorname{ad} _\xi \eta $ of $ \mathfrak{g} $ on itself will be a special case.

\begin{Definition} If $ A \rightarrow Q $ is a Lie algebroid and $ E \rightarrow Q $ is a~vector bundle, then an \emph{$A$-connection on $E$} is a bilinear map $ \nabla \colon \Gamma (A) \times \Gamma (E) \rightarrow \Gamma (E)$, $( X, u ) \mapsto \nabla _X u $, satisfying the conditions
 \begin{gather*}
 \nabla _{ f X } u = f \nabla _X u , \qquad \nabla _X ( f u ) = f \nabla _X u + \rho (X) [f] u ,
 \end{gather*}
 for all $ X \in \Gamma (A) $, $ u \in \Gamma (E) $, and $ f \in C ^\infty (Q) $.
\end{Definition}

\begin{Remark}
 A $ T Q $-connection is just an ordinary connection. Given a~$ T Q $-connection $ \nabla $ on $A$, there are two naturally-induced $A$-connections on $A$, which we write as $ \nabla $ and $ \overline{ \nabla } $:
 \begin{gather*}
 \nabla _X Y = \nabla _{\rho (X)} Y, \qquad \overline{ \nabla } _X Y = \nabla _{ \rho (Y) } X + \left[ X, Y
 \right] .
 \end{gather*}
 For example, when $ A = \mathfrak{g} \rightarrow \bullet $, the trivial $ T \bullet $-connection induces two $ \mathfrak{g} $-connections on $\mathfrak{g}$:
 \begin{gather*}
 \nabla _X Y = 0 , \qquad \overline{ \nabla } _X Y = [X, Y] = \operatorname{ad} _X Y .
 \end{gather*}
 Hence, the induced connection $ \overline{ \nabla } $ can be seen as a generalization of the adjoint action of a Lie algebra.
\end{Remark}

\begin{Definition} \label{def:pathConnection}
Let $ a \in \mathcal{P} _\rho (A) $ be an $A$-path over $ q \in \mathcal{P} (Q) $ and $ \xi \in \mathcal{P} \bigl( \Gamma (A) \bigr) $ be a time-dependent section such that $ a (t) = \xi \bigl( q (t) \bigr) $. Suppose $ u \in \mathcal{P} (E) $ has the same base path $q$, along
 with a time-dependent section $ \eta \in \mathcal{P} \bigl( \Gamma (E) \bigr) $ satisfying $ u (t) = \eta \bigl( q (t) \bigr) $. Then we def\/ine
 \begin{gather*}
 \nabla _a u (t) = \nabla _\xi \eta \bigl( t, q (t) \bigr) + \dot{\eta } \bigl( t, q (t) \bigr) ,
 \end{gather*}
which is independent of the choice of $ \xi$, $\eta $.
\end{Definition}

\begin{Definition}\label{def:admissibleVariation}
Let $ a \in \mathcal{P} _\rho (A) $ be an $A$-path over $ q \in \mathcal{P} (Q) $. An \emph{admissible variation} of $a$ is a~variation of the form $ X _{ b, a } \in T _a \mathcal{P} _\rho (A) $, where $ b \in \mathcal{P} (A) $ is a path in $A$ (but not necessarily an $A$-path!) over $q$ such that $ b (0) = 0 $ and $ b (1) = 0 $. Relative to a $ T Q $-connection $ \nabla $, the varia\-tion~$ X _{ b, a } $ has vertical component $ \overline{\nabla } _a b $ and horizontal component~$ \rho (b) $.
\end{Definition}

\begin{Remark} \label{rmk:admissibleSubbundle} Crainic and Fernandes~\cite[Proposition~4.7]{CrFe2003} show that these admissible variations form an integrable subbundle $ \mathcal{F} (A) \subset T \mathcal{P} _\rho (A) $, and the tangent subspaces $ \mathcal{F} _a (A) \subset T _a \mathcal{P} _\rho (A) $ are independent of the choice of connection $ \nabla $ in the above def\/inition.
\end{Remark}

\subsection{Lagrangian mechanics}\label{sec:algebroidLagrangianMechanics}

Now that we have appropriate paths and variations, we are prepared to discuss the variational approach to Lagrangian mechanics on Lie algebroids.

\begin{Definition} Given a Lagrangian $ L \colon A \rightarrow \mathbb{R} $, the \emph{action functional} $ S \colon \mathcal{P} _\rho (A) \rightarrow \mathbb{R} $ is def\/ined to be
 \begin{gather*}
 S (a) = \int _0 ^1 L \bigl( a (t) \bigr) \,\mathrm{d}t .
 \end{gather*}
We say that $ a \in \mathcal{P} _\rho (A) $ satisf\/ies \emph{Hamilton's variational principle for $A$-paths} if $ \mathrm{d} S ( X _{ b , a } ) = 0 $ for all admissible variations $ X _{ b, a } \in \mathcal{F} _a (A) $.
\end{Definition}

We next use the notion of admissible variation from Def\/inition~\ref{def:admissibleVariation}, and its expression in terms of a~connection on~$A$, to give a new, coordinate-free characterization of the solutions to Hamilton's variational principle for $A$-paths.

\begin{Theorem} \label{thm:connectionELP}
An $A$-path $ a \in \mathcal{P} _\rho (A) $ satisfies Hamilton's principle if and only if, given a~$ T Q $-connection $ \nabla $ on~$A$, it satisfies the differential equation
 \begin{gather}\label{eqn:connectionELP}
 \rho ^\ast \mathrm{d} L ^\textup{hor} (a) + \overline{ \nabla }_a ^\ast \mathrm{d} L ^\textup{ver} (a) = 0 ,
 \end{gather}
 where $ \mathrm{d} L ^\textup{hor} $ and $ \mathrm{d} L ^\textup{ver} $ are the horizontal and vertical components of $ \mathrm{d} L $ relative to $ \nabla $, and where~$ \rho ^\ast $ and~$ \overline{ \nabla } _a ^\ast $ are the formal adjoints of $ \rho $ and $ \overline{ \nabla } _a $.
\end{Theorem}

\begin{proof} Given an admissible variation $ X _{ b, a } \in \mathcal{F} _a (A) $, we have
 \begin{gather*}
 \mathrm{d} S ( X _{ b, a } ) = \mathrm{d} S \big( X _{ b, a } ^\text{hor} \big) + \mathrm{d} S \big( X _{ b, a } ^\text{ver} \big)
= \int _0 ^1 \Bigl( \bigl\langle \mathrm{d} L ^\text{hor} (a) , \rho (b) \bigr\rangle + \bigl\langle \mathrm{d} L ^\text{ver} (a)
 , \overline{ \nabla } _a b \bigr\rangle \Bigr) \,\mathrm{d}t \\
\hphantom{\mathrm{d} S ( X _{ b, a } )}{} = \int _0 ^1 \bigl\langle \rho ^\ast \mathrm{d} L ^\text{hor} (a)
 + \overline{ \nabla } _a ^\ast \mathrm{d} L ^\text{ver} (a) , b \bigr\rangle \,\mathrm{d}t
 \end{gather*}
 Since $b$ is arbitrary, it follows that $ \mathrm{d} S $ vanishes for all $ X _{ b, a } \in \mathcal{F} _a (A) $ if and only if $ \rho ^\ast \mathrm{d} L ^\text{hor} (a) + \overline{ \nabla } _a ^\ast \mathrm{d} L ^\text{ver} (a) $ vanishes for all~$t$.
\end{proof}

\begin{Example} \label{ex:algebraELP}
Let $ A = \mathfrak{g} \rightarrow \bullet $, where $ \mathfrak{g} $ is a Lie algebra. Any $ a \in \mathcal{P} _\rho (\mathfrak{g} ) = \mathcal{P} (\mathfrak{g} ) $ can be identif\/ied with its unique time-dependent section $ \xi (t) = \xi ( t, \bullet ) = a (t) $. Since $ \rho $ and $ \nabla $ are trivial, it follows that \eqref{eqn:connectionELP} becomes
 \begin{gather*}
 0 = \overline{ \nabla } _a ^\ast \mathrm{d} L (a) = \left( \operatorname{ad} _\xi + \frac{\mathrm{d}}{\mathrm{d}t}
 \right) ^\ast \frac{ \delta L }{ \delta \xi } = \left( \operatorname{ad} _\xi ^\ast - \frac{\mathrm{d}}{\mathrm{d}t} \right) \frac{ \delta L }{ \delta \xi },
 \end{gather*}
 which are precisely the \emph{Euler--Poincar\'e equations} (cf.\ Marsden and Ratiu~\cite[Chapter~13]{MaRa1999}).
\end{Example}

Next, we show that this coordinate-free formulation agrees with the local-coordinate expression obtained by Weinstein~\cite{Weinstein1996} for regular Lagrangians and by Mart\'{\i}nez~\cite{Martinez2001,Martinez2007,Martinez2008} in the more general case.

\begin{Theorem} \label{thm:coordinateELP} Let $ q ^i $ be local coordinates for $Q$, $ \{ e _I \} $ be a local basis of sections of~$A$, and~$ \nabla $ the locally trivial $ T Q $-connection defined by $ \nabla _{ \partial / \partial q ^i } e _I \equiv 0 $. Let $ \rho ^i _I $ and $ C _{ I J } ^K $ be the local-coordinate representations of $ \rho $ and $ [\cdot , \cdot ] $, where
 \begin{gather*}
 \rho ( e _I ) = \rho ^i _I \frac{ \partial }{ \partial q ^i } , \qquad [ e _I , e _J ] = C _{ I J } ^K e _K .
 \end{gather*}
If $ a \in \mathcal{P} (A) $ has the local-coordinate representation $ a (t) = \xi ^I (t) e _I \bigl( q (t) \bigr) $, then $a$ is an $A$-path if and only if $ \dot{q} ^i = \rho ^i _I \xi ^I $, and $a$ satisfies~\eqref{eqn:connectionELP} if and only if
 \begin{gather} \label{eqn:coordinatesELP}
 \rho ^i _I \frac{ \partial L }{ \partial q ^i } - C ^K _{ I J } \xi ^J \frac{ \partial L }{ \partial \xi ^K } - \frac{\mathrm{d}}{\mathrm{d}t} \frac{ \partial L }{ \partial \xi ^I } = 0 .
 \end{gather}
\end{Theorem}

\begin{proof}
 For the $A$-path condition, we have
 \begin{gather*}
 \dot{q} = \dot{q} ^i \frac{ \partial }{ \partial q ^i }, \qquad \rho(a) = \rho \big(\xi ^I e _I \big) = \rho ^i _I \xi ^I \frac{ \partial }{ \partial q ^i },
 \end{gather*}
 so these are equal if and only the $ \partial / \partial q ^i $ coef\/f\/icients are equal. Next, the horizontal and vertical components of $ \mathrm{d} L $ are
 \begin{gather*}
 \mathrm{d} L ^\text{hor} = \frac{ \partial L }{ \partial q ^i } \,\mathrm{d}q ^i , \qquad \mathrm{d} L ^\text{ver}= \frac{ \partial L }{ \partial \xi ^I } e ^I ,
 \end{gather*}
 where, as usual, $ e ^I $ is the dual basis element satisfying $ e ^I e _J = \delta ^I _J $. Moreover, extending~$a$ to the time-dependent section $ \xi (t) = \xi ^J (t) e _J $, we have
 \begin{gather*}
 \overline{ \nabla } _a \eta ^I e _I = \overline{ \nabla } _{ \xi ^J e _J } \eta ^I e _I + \dot{ \eta } ^I e _I= \big[ \xi ^J e _J , \eta ^I e _I \big] + \dot{ \eta } ^I e _I = - C _{ I J } ^K \xi ^J \eta ^I e _K + \dot{ \eta } ^I e _I ,
 \end{gather*}
 so $ \overline{ \nabla } _a = - C _{ I J } ^K \xi ^J e ^I e _K + \mathrm{d}/\mathrm{d}t $. Finally,
 \begin{gather*}
 \rho ^\ast \mathrm{d} L ^\text{hor} + \overline{ \nabla } _a ^\ast
 \mathrm{d} L ^\text{ver} = \rho ^i _I \frac{ \partial L }{ \partial q ^i } e ^I - C _{ I J } ^K \xi ^J \frac{ \partial L
 }{ \partial \xi ^K } e ^I - \frac{\mathrm{d}}{\mathrm{d}t} \frac{ \partial L }{ \partial \xi ^I } e ^I ,
 \end{gather*}
 so the left side vanishes if and only if all the $ e ^I $ coef\/f\/icients on the right side vanish, i.e., \eqref{eqn:connectionELP} holds if and only if~\eqref{eqn:coordinatesELP} holds.
\end{proof}

\begin{Example} \label{ex:tangentELP}
Suppose $ A = T Q \rightarrow Q $. Local coordinates $ q ^i $ on $Q$ yield corresponding local sections $ \partial / \partial q ^i $ of~$TQ$, i.e., $ e _i = \partial / \partial q ^i $. It follows that $ [ e _i , e _j ] \equiv 0 $ and thus $ C ^k _{ i j } \equiv 0 $ for all $ i$, $j$, $k $. Since $ \rho $ is the identity map, we have $ \rho ^i _j = \delta ^i _j $, so the $ T Q $-path condition is $ \dot{q} ^i = \xi ^i $. Putting this all together, it follows that~\eqref{eqn:coordinatesELP} yields
 \begin{gather*}
 \frac{ \partial L }{ \partial q ^i } - \frac{\mathrm{d}}{\mathrm{d}t} \frac{ \partial L }{ \partial \dot{q} ^i } = 0 ,
 \end{gather*}
 i.e., the ordinary Euler--Lagrange equations.
\end{Example}

\begin{Remark} \label{rmk:poisson} There is also an equivalent symplectic/pre-symplectic/Poisson approach to Lagrangian mechanics on Lie algebroids, which has already been well studied in previous work on the subject.

Mart\'{\i}nez~\cite{Martinez2001} shows that one can def\/ine a Lie algebroid notion of dif\/ferential forms (just as we did for the vertical formalism in Section~\ref{sec:fibered}), as well as a version of the tautological $1$-form and canonical $2$-form on $ A ^\ast $. The Legendre transform $ \mathbb{F} L = \mathrm{d} L ^\text{ver} \colon A \rightarrow A ^\ast $ is then used to pull this back to a Lagrangian $2$-form on~$A$ (in the sense of forms on Lie algebroids) and to def\/ine an energy function $ E _L $ on $A$, which Mart\'{\i}nez~\cite{Martinez2001} uses to obtain Lagrangian dynamics on~$A$.

Weinstein~\cite{Weinstein1996}, on the other hand, uses the canonical Poisson structure on $ A ^\ast $ (which generalizes the Lie--Poisson structure on the dual of a Lie algebra), which can be pulled back along $ \mathbb{F} L $ to $ A $ when $L$ is a regular Lagrangian. In this case, the Poisson structure on $A$ induces a Lagrangian vector f\/ield associated to $ E _L $ in the usual way.

The approach of Grabowska et al.~\cite{GrUrGr2006}, Grabowska and Grabowski~\cite{GrGr2008} extends Weinstein's approach in a dif\/ferent direction: instead of using the canonical Poisson structure on $ A ^\ast $, which maps $ T ^\ast A ^\ast \rightarrow T A ^\ast $, they use a related map $ \epsilon \colon T ^\ast A \rightarrow T A ^\ast $ to def\/ine the \emph{Tulczyjew differential} $ \Lambda _L = \epsilon \circ \mathrm{d} L \colon A \rightarrow T A ^\ast $. (The map $\epsilon$ is related to the canonical Poisson map by the Tulczyjew isomorphism $ T ^\ast A ^\ast \xrightarrow{\cong} T ^\ast A $.) Using this framework, one requires that $ a \in \mathcal{P} (A) $ satisfy $ \frac{\mathrm{d}}{\mathrm{d}t} \mathbb{F} L (a) = \Lambda _L (a) $, which contains the Euler--Lagrange--Poincar\'e equations together with the $A$-path condition. We remark that Grabowska et al.~\cite{GrUrGr2006}, Grabowska and Grabowski~\cite{GrGr2008} apply this approach both to Lie algebroids and to so-called ``general algebroids,'' for which the map $\epsilon$ is taken as primitive, and where there is generally no canonical Poisson structure on the dual.
\end{Remark}

\subsection{Special case: the Lagrange--Poincar\'e equations} \label{sec:lagrangePoincare}

The \emph{Lagrange--Poincar\'e equations} on a principal bundle $ Q \rightarrow Q / G $ are typically derived by the procedure of \emph{Lagrangian reduction} (cf.\ Marsden and Scheurle~\cite{MaSc1993b}, Cendra et al.~\cite{CeMaRa2001}), relative to a particular choice of principal connection. We now discuss how these equations may instead be obtained directly on the Atiyah algebroid $ A = T Q / G \rightarrow Q / G $, using the framework presented above, and how the choice of principal connection is related to the connection $ \nabla $ on $A$. (Note that $ Q / G $, not $Q$, is the base of this algebroid.) In particular,
Example~\ref{ex:algebraELP} corresponds to the case $ Q = G $, while Example~\ref{ex:tangentELP} corresponds to the case where $G$ is trivial.

Let $ L \colon T Q / G \rightarrow \mathbb{R} $ be a Lagrangian on the Atiyah algebroid. A principal connection corresponds to a section of the anchor $ \rho \colon T Q / G \rightarrow T ( Q / G ) $, i.e., a~right splitting of the \emph{Atiyah sequence},
\begin{gather} \label{eqn:atiyahSequence}
 0 \rightarrow \widetilde{ \mathfrak{g} } \rightarrow T Q / G \xrightarrow{ \rho } T ( Q / G ) \rightarrow 0 .
 \end{gather}
Here, following Cendra et al.~\cite{CeMaRa2001}, we use $ \widetilde{ \mathfrak{g} } $ to denote the adjoint bundle $ Q \times _G \mathfrak{g} $, so a left splitting is a principal connection $1$-form (cf.\ Mackenzie~\cite[Chapter~5]{Mackenzie2005}). This splitting lets us write $ T Q / G \cong T ( Q / G ) \oplus \widetilde{ \mathfrak{g} } $; the anchor $\rho$ is just projection onto the f\/irst component, and the bracket of two sections $ \xi = ( X, \overline{ \xi } ) $ and $ \eta = ( Y , \overline{ \eta } ) $ is
 \begin{gather} \label{eqn:atiyahBracket}
 \bigl[ ( X, \overline{ \xi } ) , ( Y , \overline{ \eta } ) \bigr] = \bigl( [X, Y ] , \widetilde{ \nabla } _X \overline{ \eta } - \widetilde{ \nabla } _Y \overline{ \xi } + [ \overline{ \xi } , \overline{ \eta } ] - \widetilde{ R } ( X, Y ) \bigr) ,
 \end{gather}
 where $ \widetilde{ \nabla } $ is the covariant derivative and $ \widetilde{ R } $ the curvature form of the principal connection (cf.\ Cendra et al.~\cite[Theorem~5.2.4]{CeMaRa2001} in this particular case and Mackenzie~\cite[Theorem~7.3.7]{Mackenzie2005} in a~more general setting).

 Relative to the splitting induced by the principal connection, $A$-paths have the form $ a = ( x, \dot{x} , \overline{ v } ) $, where $x$ is the base path in $ Q / G $. As before, we extend $a$ to a time-dependent section $ \xi = ( X, \overline{ \xi } ) $, and likewise, we extend an arbitrary path $b = ( x, \delta x , \overline{ w } ) $ to a time-dependent section $ \eta = ( Y, \overline{ \eta } ) $. To f\/ind the corresponding admissible variation $ \delta a $, we calculate $ \rho (b) = \delta x $ and use~\eqref{eqn:atiyahBracket} to obtain
 \begin{gather*}
 \overline{ \nabla } _a b = \overline{ \nabla } _\xi \eta + \dot{ \eta }
= \overline{ \nabla } _{ ( X, \overline{ \xi } ) } ( Y, \overline{ \eta } ) + ( \dot{ Y }, \dot{ \overline{ \eta } } ) = \nabla _Y ( X, \overline{ \xi } ) + \bigl[ ( X, \overline{ \xi } ), (Y, \overline{ \eta } ) \bigr] + ( \dot{ Y } , \dot{ \overline{ \eta } } ) \\
\hphantom{\overline{ \nabla } _a b}{} = \bigl( \nabla _Y X + [ X, Y ] + \dot{ Y } , \widetilde{ \nabla } _X \overline{ \eta } + [ \overline{ \xi }, \overline{ \eta } ] - \widetilde{ R } ( X, Y ) + \dot{ \overline{ \eta } } \bigr) \\
\hphantom{\overline{ \nabla } _a b}{}= \bigl( \overline{ \nabla } _X Y + \dot{ Y } , (\widetilde{ \nabla } _X \overline{ \eta } + \dot{ \overline{ \eta } }) + [ \overline{ \xi }, \overline{ \eta } ] - \widetilde{ R } ( X, Y ) \bigr) \\
\hphantom{\overline{ \nabla } _a b}{} = \bigl( \overline{ \nabla } _{ \dot{x} } (\delta x) , \widetilde{ \nabla } _{\dot{x}} \overline{ w } + [ \overline{ v }, \overline{ w } ] - \widetilde{ R } ( \dot{x}, \delta x ) \bigr).
 \end{gather*}
 (Here, we chose $ \nabla $ to be compatible with $ \widetilde{ \nabla } $, so that the $ \nabla _Y \overline{ \xi } $ and $ \widetilde{ \nabla } _Y \overline{ \xi } $ terms cancel.) Therefore, admissible variations have the form $ \delta a = ( \delta x, \delta \dot{x}, \delta \overline{ v } ) $, where
 \begin{gather*}
 \delta \overline{ v } = \widetilde{ \nabla } _{\dot{x}} \overline{ w } + [ \overline{ v }, \overline{ w } ] - \widetilde{ R } ( \dot{x}, \delta x ) ,
 \end{gather*}
 and these are precisely the admissible variations of Cendra et al.~\cite[Theorem~3.4.1]{CeMaRa2001}.

 Furthermore, now that we have expressions for $ \rho $ and $ \overline{ \nabla } $ in terms of the splitting induced by the principal connection, it is a straightforward matter to write down the Euler--Lagrange--Poincar\'e equations \eqref{eqn:connectionELP} in terms of their adjoints. If we write $ L = L ( x, \dot{x}, \overline{ v } ) $, then
 \begin{gather*}
 \rho ^\ast \mathrm{d} L ^{ \text{hor} } ( x, \dot{x}, \overline{ v } ) + \overline{ \nabla } _{ ( x, \dot{x}, \overline{ v } ) } ^\ast \mathrm{d} L ^{ \text{ver} } ( x, \dot{x}, \overline{ v } ) \\
 \qquad{} = \left( \frac{ \partial L }{ \partial x } + \overline{ \nabla } _{ \dot{x} } ^\ast \frac{ \partial L }{ \partial \dot{x} } - ( i
 _{\dot{x}} \widetilde{ R } ) ^\ast \frac{ \partial L }{ \partial \overline{ v } } \right) \mathrm{d}x + \left( \widetilde{
 \nabla } _{ \dot{x} } ^\ast \frac{ \partial L }{ \partial \overline{ v } } + \operatorname{ad} _{ \overline{ v } } ^\ast
 \frac{ \partial L }{ \partial \overline{ v } } \right) \mathrm{d} \overline{ v } .
 \end{gather*}
 Hence, this vanishes precisely when
 \begin{gather*}
 \frac{ \partial L }{ \partial x } + \overline{ \nabla } _{ \dot{x} } ^\ast \frac{ \partial L }{ \partial \dot{x} } - ( i _{\dot{x}} \widetilde{ R } ) ^\ast \frac{ \partial L }{ \partial \overline{ v } } = 0 , \qquad \widetilde{ \nabla } _{ \dot{x} } ^\ast \frac{ \partial L }{ \partial
 \overline{ v } } + \operatorname{ad} _{ \overline{ v } } ^\ast \frac{ \partial L }{ \partial \overline{ v } } = 0 ,
 \end{gather*}
which are exactly the coordinate-free Lagrange--Poincar\'e equations of Cendra et al.~\cite[Theorem~3.4.1]{CeMaRa2001}. (The only notable dif\/ference
 in notation is that Cendra et al.~\cite{CeMaRa2001} write both covariant derivatives $ \overline{ \nabla } _{ \dot{x} } $ and $ \widetilde{ \nabla } _{ \dot{x} } $ as $ {D} / {D} t $ and their adjoints as $ - {D} / {D} t $.)

\begin{Remark} The argument above works not only for the Atiyah algebroid of a~principal bundle, but also in the more general setting discussed in Mackenzie~\cite[Chapter~7]{Mackenzie2005}, where one can split a short exact sequence similar to~\eqref{eqn:atiyahSequence} and obtain a~bracket of the form~\eqref{eqn:atiyahBracket}. This includes the so-called \emph{transitive Lie algebroids}, of which the Atiyah algebroid is a particular example.
\end{Remark}

\begin{Example} Wong's equations \cite{Wong1970} for a particle in a Yang--Mills f\/ield are a classic example of Lagrange--Poincar\'e theory. Following the presentation in Cendra et al.~\cite[Chapter~4]{CeMaRa2001}, we suppose that $ Q \rightarrow Q / G $ is a principal $G$-bundle equipped with a Riemannian metric $g$ on the base $ Q / G $ and a~bi-invariant Riemannian metric $\kappa$ on the structure group~$G$. Using a principal connection to split
 $ T Q / G \cong T ( Q / G ) \oplus \widetilde{ \mathfrak{g} } $, and denoting by $k$ the f\/iber metric on $ \widetilde{ \mathfrak{g} } $ corresponding to $\kappa$, we take the Lagrangian
 \begin{gather*}
 L ( x, \dot{x}, \overline{ v } ) = \frac{1}{2} k ( \overline{ v} , \overline{ v } ) + \frac{1}{2} g ( \dot{x} , \dot{x} ) .
 \end{gather*}
The af\/f\/ine connection $ \nabla $ is then chosen to agree with $ \widetilde{ \nabla } $ on $ \widetilde{ \mathfrak{g} } $ and with the Levi-Civita connection associated to $g$ on the base.

With this connection in hand, we now compute the $ \mathrm{d} L ^{ \text{ver} } $ components,
 \begin{gather*}
 \frac{ \partial L }{ \partial \dot{x} } = g ( \dot{x}, \cdot ) = g ^\flat ( \dot{x} ) , \qquad \frac{ \partial L }{ \partial \overline{ v } } = k ( \overline{ v } , \cdot ) = k ^\flat ( \overline{ v } ) ,
 \end{gather*}
using the familiar ``f\/lat'' notation for metrics. Since the f\/iber metric $k$ is necessarily $ \operatorname{ad} $-invariant, the term $ \operatorname{ad} ^\ast _{ \overline{ v } } k ^\flat ( \overline{ v } ) $ vanishes, so the $ \mathrm{d} \overline{ v } $ component of the Lagrange--Poincar\'e equations is
 \begin{gather} \label{eqn:wong1}
 \widetilde{ \nabla } _{ \dot{x} } k ^\flat ( \overline{ v } ) = 0 .
 \end{gather}
 Next, since $ \nabla $ agrees with the Levi-Civita connection on $ Q / G $, the torsion-free property implies
 \begin{gather*}
 \overline{ \nabla } _X Y = \nabla _Y X + [X, Y ] = \nabla _X Y ,
 \end{gather*}
so we just have $ \overline{ \nabla } \equiv \nabla $. Moreover, using the metric-compatibility of $ \nabla $ along with~\eqref{eqn:wong1} to compute $ \mathrm{d} L ^{ \text{hor} } $, it can be seen that
 \begin{gather*}
 \frac{ \partial L }{ \partial x } + \overline{ \nabla } _{ \dot{x} } ^\ast \frac{ \partial L }{ \partial \dot{x} } = g ^\flat ( \nabla _{ \dot{x} } \dot{x} ) ,
 \end{gather*}
 and therefore the $ \mathrm{d} x $ component of the Lagrange--Poincar\'e equations is
 \begin{gather} \label{eqn:wong2}
 g ^\flat ( \nabla _{ \dot{x} } \dot{x} ) = \big( i _{ \dot{x} } \widetilde{ R } \big) ^\ast k ^\flat ( \overline{ v } ) .
 \end{gather}
The equations \eqref{eqn:wong1} and \eqref{eqn:wong2} are precisely the coordinate-free version of Wong's equations. For further discussion on Wong's equations from the perspective of Lie algebroids, see Le\'on et al.~\cite{LeMaMa2005}, Grabowska et al.~\cite{GrUrGr2006}.

 We conclude this example with some remarks on the relationship between Wong's equations and the generalized notion of geodesics on a~Lie algebroid. Montgomery~\cite{Montgomery1990} called $ g \oplus k $ a~\emph{Kaluza--Klein metric} and related Wong's equations to \emph{Kaluza--Klein geodesics}. However, a Kaluza--Klein metric is a particular example of a \emph{Lie algebroid metric} (in this case, on $ A = T Q / G $), for which there is a unique Levi-Civita
 (torsion-free, metric-compatible) $A$-connection $ \nabla $, and one may consider the corresponding geodesic equations,
 \begin{gather*}
 \nabla _a a = 0 .
 \end{gather*}
(See Crainic and Fernandes~\cite{CrFe2003}, Cort\'es and Mart\'{\i}nez~\cite{CoMa2004}, Cort\'es et al.~\cite{CoLeMaMaMa2006}.) Grabowska et al.~\cite{GrUrGr2006} pointed out that Wong's equations may in fact be considered a~special case of the generalized geodesic equations on a~Lie algebroid; this correspondence is hidden slightly by the fact that Wong's equations are written relative to an $A$-connection obtained from $ \widetilde{ \nabla } $ rather than the Levi-Civita $A$-connection.
\end{Example}

\subsection{Fibered manifolds revisited}

The results of Section~\ref{sec:EL} for f\/ibered manifolds are, in fact, a~special case of Lagrangian mechanics on the Lie algebroid $ V Q $.

Recall from Example~\ref{ex:tangentAlgebroid} that, whenever $ Q \rightarrow M $ is a f\/ibered manifold, the vertical bundle~$ V Q $ is a Lie algebroid over $Q$; in particular, it is a Lie subalgebroid of $ T Q $, from which it inherits the bracket $ [ \cdot , \cdot ] $, projection $ \rho $, and (identity) anchor $ \rho $. Now, by Def\/inition~\ref{def:a-path}, $ a \in \mathcal{P} ( V Q ) $ over $ q \in \mathcal{P} (Q) $ is a $ V Q $-path if and only if it
satisf\/ies $ \dot{q} = a $. Since $ a (t) \in V Q $ for each $ t \in I $, this means that $ V Q $-paths are precisely the tangent prolongations of vertical paths $ q \in \mathcal{P} _V (Q) $. Hence, we may identify $ \mathcal{P} _\rho ( V Q ) $ with~$ \mathcal{P} _V (Q) $.

Suppose now that $ L \colon V Q \rightarrow \mathbb{R} $ is a~Lagrangian in the sense of Section~\ref{sec:algebroidLagrangianMechanics}. If $ ( x ^\sigma , y ^i ) $ are f\/iber-adapted local coordinates for $Q \rightarrow M $, then $ e _i = \partial / \partial y ^i $ def\/ines a~basis of local sections of~$ V Q $. Since an $A$-path is just a~tangent prolongation of a vertical path, it follows that the $A$-path conditions are $ \dot{x} ^\sigma = 0 $ and $ \dot{y} ^i = \xi ^i $. Furthermore, as in Example~\ref{ex:tangentELP}, we have $ \rho ^i _j = \delta ^i _j $, $ \rho ^i _\sigma \equiv 0 $, and $ C _{ i j } ^k \equiv 0 $, so~\eqref{eqn:coordinatesELP} becomes
\begin{gather*}
 \frac{ \partial L }{ \partial y ^i } - \frac{\mathrm{d}}{\mathrm{d}t} \frac{ \partial L }{ \partial \dot{y} ^i } = 0 .
\end{gather*}
Together with the $A$-path condition, this agrees precisely with the vertical Euler--Lagrange equations~\eqref{eqn:verticalEL}.

\subsection{Lie algebroid morphisms and reduction}

Finally, we give a brief review of Lagrangian reduction on Lie algebroids. Weinstein~\cite{Weinstein1996} and Mart\'{\i}nez~\cite{Martinez2008} showed that, whenever $ \Phi \colon A \rightarrow A ^\prime $ is a Lie algebroid morphism, then one can relate Lagrangian dynamics on~$A$ to those on $ A ^\prime $.

Informally, a Lie algebroid morphism is a mapping that ``preserves'' the Lie algebroid structure in an appropriate sense. More precisely, if $ A \rightarrow M $ and $ A ^\prime \rightarrow M ^\prime $ are Lie algebroids (possibly over dif\/ferent base manifolds), then a bundle mapping $ \Phi \colon A \rightarrow A ^\prime $ is a \emph{Lie algebroid morphism} if the dual comorphism $ \Phi ^\ast \colon A ^{\prime \ast} \rightarrow A ^\ast $ is a~Poisson relation with respect to the canonical Poisson structures on $ A ^\ast $ and $ A ^{\prime \ast} $. (See also
Remark~\ref{rmk:poisson}.)

\begin{Theorem} \label{thm:weinsteinMartinez} Let $ \Phi \colon A \rightarrow A ^\prime $ be a morphism of Lie algebroids, and suppose $ L \colon A \rightarrow \mathbb{R} $ and $ L ^\prime \colon A ^\prime \rightarrow \mathbb{R} $ are Lagrangians such that $ L = L ^\prime \circ \Phi $. If $ a \in \mathcal{P} _\rho (A) $ is such that $ a ^\prime = \Phi \circ a \in \mathcal{P} _{ \rho ^\prime } ( A ^\prime ) $ is a solution path for $ L ^\prime $, then $a$ is a solution path
 for $L$. Moreover, the following converse holds when $ \Phi \colon A \rightarrow A ^\prime $ is fiberwise surjective: If $ a \in \mathcal{P} _\rho (A) $ is a solution path for $L$, then $ a ^\prime = \Phi \circ a \in \mathcal{P} _{ \rho ^\prime } ( A ^\prime ) $ is a solution path for $ L ^\prime $.
\end{Theorem}

\begin{proof} See Mart\'{\i}nez~\cite[Theorems 5--6]{Martinez2008}. This generalized results by Weinstein~\cite[Theorems~4.8 and~4.5, respectively]{Weinstein1996} for regular Lagrangians, where the converse also required the stronger assumption that $\Phi$ be a f\/iberwise isomorphism.
\end{proof}

For example, if $G$ is a Lie group acting freely and properly on $Q$, then the quotient morphism $ T Q \rightarrow T Q / G $ is a Lie algebroid morphism, and the corresponding reduction theory is just classical Lagrangian reduction. However, there is a much more general class of quotient morphisms~-- for f\/ibered manifolds~-- that bear directly on reduction theory, and this is the topic of the next section.

\section{Lie groupoid symmetries and reduction on f\/ibered manifolds}\label{sec:reduction}

In this section, we recall the def\/inition of a \emph{Lie groupoid} $ G \rightrightarrows M $ and of a free, proper Lie groupoid action on a f\/ibered manifold $ Q \rightarrow M $ over the same base manifold. We then show that there is a quotient morphism $ V Q \rightarrow V Q / G $, which is a Lie algebroid morphism, and hence applying Theorem~\ref{thm:weinsteinMartinez} yields a reduction theory for f\/ibered Lagrangian mechanics. This generalizes the special
case $ M = \bullet $, in which $G$ is a Lie group acting on an ordinary manifold $Q$ and the quotient morphism $ T Q \rightarrow T Q / G $ is the one used in ordinary Lagrangian reduction.

\subsection{Lie groupoids}

Just as it is natural to consider Lie group actions on ordinary manifolds, it is natural to consider \emph{Lie groupoid} actions on f\/ibered manifolds. We begin by recalling the def\/inition of a Lie groupoid and a groupoid action, as well as giving a few examples. We then prove that, just as a free and proper Lie group action on an ordinary manifold $Q$ lifts to $ T Q $, so, too, does a free and proper Lie groupoid action on a f\/ibered manifold $ Q \rightarrow M $
lift to $ V Q $.

\begin{Definition} \label{def:groupoid} A \emph{groupoid} is a small category in which every morphism is invertible. Specif\/ically, a groupoid denoted
 $ G \rightrightarrows M $ consists of a space of morphisms $G$, a~space of objects $M$, and the following structure maps:
\begin{enumerate}[label=(\roman*)]\itemsep=0pt
\item a \emph{source map} $ \alpha \colon G \rightarrow M $ and \emph{target map} $ \beta \colon G \rightarrow M $;
\item a \emph{multiplication map} $ m \colon G \mathbin{_\alpha \times _\beta } G \rightarrow G $, $ ( g, h ) \mapsto g h $;
\item an \emph{identity section} $ \epsilon \colon M \rightarrow G $, such that for all $ g \in G $,
 \begin{gather*}
 g \epsilon ( \alpha (g) \bigr) = g = \epsilon \bigl( \beta (g)\bigr) g ;
 \end{gather*}
\item and an \emph{inversion map} $ i \colon G \rightarrow G $, $ g \mapsto g ^{-1} $, such that for all $ g \in G $,
 \begin{gather*}
 g ^{-1} g = \epsilon \bigl( \alpha (g) \bigr) , \qquad g g ^{-1} = \epsilon \bigl( \beta (g) \bigr) .
 \end{gather*}
\end{enumerate}
A \emph{Lie groupoid} is a groupoid $ G \rightrightarrows M $ where $G$ and $M$ are smooth manifolds, $\alpha$ and $\beta$ are submersions, and $m$ is smooth.
\end{Definition}

\begin{Remark} A few other properties of the structure maps are immediate from this def\/inition of a Lie groupoid: in particular, it also follows that
 $m$ is a submersion, $\epsilon$ is an immersion, and $i$ is a~dif\/feomorphism.
\end{Remark}

\begin{Example} \label{ex:lieGroup} A Lie group is a Lie groupoid $ G \rightrightarrows \bullet $ over a~single point.
\end{Example}

\begin{Example} \label{ex:pairGroupoid} If $Q$ is a smooth manifold, then the \emph{pair groupoid} $ Q \times Q \rightrightarrows Q $, def\/ined by the structure maps
 \begin{gather*}
 \alpha ( q _1 , q _0 ) = q _0 , \qquad \beta ( q _1 , q _0 ) = q _1 , \qquad m \bigl( ( q _2 , q _1 ) , ( q _1 , q _0 ) \bigr) = ( q _2 , q _0 ) ,\\
 \epsilon (q) = ( q, q ) , \qquad i ( q _1 , q _0 ) = ( q _0, q _1 ) ,
 \end{gather*}
 is a Lie groupoid. More generally, if $ \mu \colon Q \rightarrow M $ is a f\/ibered manifold and
 \begin{gather*}
 Q \mathbin{ _\mu \times _\mu } Q = \bigl\{ ( q _1, q _0 ) \in Q \times Q\colon \mu ( q _1 ) = \mu ( q _2 ) \bigr\},
 \end{gather*}
 then $ Q \mathbin{ _\mu \times _\mu } Q \rightrightarrows Q $ is also a Lie groupoid, and its structure maps are just the restrictions of those above for $ Q \times Q \rightrightarrows Q $. We then say that $ Q \mathbin{ _\mu \times _\mu } Q \rightrightarrows Q $ is a~\emph{Lie subgroupoid} of $ Q \times Q \rightrightarrows Q $.
\end{Example}

\begin{Example} \label{ex:gaugeGroupoid} Let $G$ be a Lie group and $ Q \rightarrow Q / G $ be a principal $G$-bundle, i.e., $G$ acts freely and properly on $Q$. The diagonal action of $G$ on $ Q \times Q $ is also free and proper, so we may form the quotient $ ( Q \times Q ) / G $. Let $ [ q ] \in Q / G $
 denote the orbit of $ q \in Q $ and $ [ q _1 , q _0 ] \in ( Q \times Q ) / G $ denote the orbit of $ ( q _1 , q _0 ) \in Q \times Q $. Then the \emph{gauge groupoid} (or \emph{Atiyah groupoid}) $ ( Q \times Q ) / G \rightrightarrows Q / G $ of the principal bundle is def\/ined by the structure maps
 \begin{gather*}
 \alpha \bigl( [ q _1 , q _0 ] \bigr) = [q _0] , \qquad \beta \bigl( [ q _1, q _0 ] \bigr) = [q _1] , \qquad m \bigl( [q _2 , q
 _1 ] , [ q _1 , q _0 ] \bigr) = [ q _2 , q _0 ] ,\\
 \epsilon \bigl( [q] \bigr) = [ q, q ] , \qquad i \bigl( [ q _1 , q _0] \bigr) = [q _0, q _1].
 \end{gather*}
 Notice that $ G \rightrightarrows \bullet $ is the special case where $ Q = G $ acts on itself by multiplication, while $ Q \times Q \rightrightarrows Q $ is the special case where $ G = \{ e \} $ acts trivially on~$Q$.
\end{Example}

\begin{Definition} \label{def:groupoidAction}
A \emph{left action} (or just \emph{action}) of a Lie groupoid $ G \rightrightarrows M $ on a f\/ibered manifold $ Q \rightarrow M $ is a smooth map $ G \mathbin{ _\alpha \times _\mu } Q \rightarrow Q $, $ ( g, q ) \mapsto g q $, such that
 \begin{enumerate}[label=(\roman*)]\itemsep=0pt
 \item $ \mu ( g q ) = \beta (g) $ for all $ ( g, q ) \in G \mathbin{ _\alpha \times _\mu } Q $,
 \item $ g ( h q ) = ( g h ) q $ for all $ ( g, h , q ) \in G \mathbin{ _\alpha \times _\beta } G \mathbin{ _\alpha \times _\mu } Q $, and
 \item $ \epsilon \bigl( \mu (q) \bigr) q = q $ for all $ q \in Q $.
 \end{enumerate}
 The action is \emph{free} if $ g q = q $ implies $ g = \epsilon \bigl( \mu (q) \bigr) $, and it is \emph{proper} if its graph,
 \begin{gather*}
 G \mathbin{ _\alpha \times _\mu } Q \rightarrow Q \times Q , \qquad ( g, q ) \mapsto ( g q, q ) ,
 \end{gather*}
is a proper map. A \emph{principal $G$-space} is a f\/ibered manifold endowed with a free and proper $G$-action.
\end{Definition}

\begin{Remark} \label{rmk:freeProperQuotient}
As with group actions, it can be shown that if $G \rightrightarrows M $ acts freely and properly on $Q \rightarrow M $, then the quotient $ Q / G $ consisting of $G$-orbits is a smooth manifold, and there is a~smooth quotient map $ Q \rightarrow Q / G $. We refer to Dufour and Zung~\cite[Chapter~7]{DuZu2005} for a~more detailed discussion of this and other properties of groupoid actions.
\end{Remark}

\begin{Example} \label{ex:groupAction}The action of a Lie group $G$ on a manifold $Q$ is precisely the action of the Lie groupoid $ G \rightrightarrows \bullet $ on the f\/ibered manifold $ Q \rightarrow \bullet $. If the action is free and proper, then the associated principal $G$-space corresponds to the principal $G$-bundle $Q \rightarrow Q /G$.
\end{Example}

\begin{Example} \label{ex:pairAction}
For any smooth manifold $Q$, the pair groupoid $ Q \times Q \rightrightarrows Q $ acts on $Q$ by $ ( q _1 , q _0 ) q _0 = q _1 $. (In this case, we treat $Q$ as the f\/ibered manifold $ Q \rightarrow Q $, rather than $ Q \rightarrow \bullet $.) Since any two points $ q _0$, $q _1 $ lie in the same orbit, it follows that $ Q / ( Q \times Q ) \cong \bullet $, and the quotient map is simply $ Q \rightarrow \bullet $.
\end{Example}

\begin{Example} Let $G$ be a Lie group acting freely and properly on $Q$, so that $ Q \rightarrow Q / G $ is a~principal $G$-bundle. Then the gauge groupoid $ ( Q \times Q ) / G $ acts on $ Q \rightarrow Q / G $, in the sense of Def\/inition~\ref{def:groupoidAction}, and is uniquely def\/ined by the condition $ [q _1, q _0 ] q _0 = q _1 $. (Notice that Example~\ref{ex:pairAction} is the special case where $ G = \{ e \} $ acts trivially on $Q$.) Again, we see that any two points $ q _0 , q _1 \in Q $ lie in the same orbit, so $ Q / \bigl( ( Q \times Q ) / G \bigr) \cong \bullet $, and the quotient map is $ Q \rightarrow \bullet $.
\end{Example}

\begin{Example} \label{ex:multiplicationAction} For any Lie groupoid $ G \rightrightarrows M $, the multiplication map $m$ is an action of~$G$ on itself, treated as the f\/ibered manifold $ \beta \colon G \rightarrow M $. This action is free, since $ g h = h $ implies $ g = (g h) h ^{-1} = h h ^{-1} = \epsilon \bigl( \beta (h) \bigr) $. Moreover, the action is proper: $( g , h ) \mapsto ( gh , h ) $ is a~dif\/feomorphism, having the inverse $ ( g, h ) \mapsto ( g h ^{-1} , h ) $, so in particular it is a proper map.

The orbit of each $ h \in G $ is its $\alpha$-f\/iber $ \alpha ^{-1} \bigl( \{ x \} \bigr) $, where $ x = \alpha (h) $. Identifying the f\/iber $ \alpha ^{-1} \bigl( \{ x \} \bigr) $ with the corresponding base point $ x \in M $, it follows that $ G / G \cong M $, and the quotient map is just $ \alpha \colon G \rightarrow M $.
\end{Example}

\begin{Example} If $ G \rightrightarrows M $ acts on $ Q \rightarrow M $, then it also acts on $ V Q \rightarrow M $, considered as a~f\/ibered manifold. Specif\/ically, we have the action
 \begin{gather*}
 G \mathbin{ _\alpha \times _{\mu \circ \tau}} V Q \rightarrow V Q , \qquad ( g , v ) \mapsto g _\ast v ,
 \end{gather*}
 where $ g _\ast $ denotes the pushforward of $ q \mapsto gq $.
\end{Example}

\begin{Lemma} \label{lem:quotientGroupoid}
Suppose $ G \rightrightarrows M $ has a free, proper action on $ Q \rightarrow M $. Then its diagonal action on $ Q \mathbin{ _\mu \times _\mu } Q \rightarrow M $, given by $ g ( q _1 , q _0 ) = ( g q _1 , g q _0 ) $, is also free and proper. Moreover, the quotient can be given a natural Lie group
 structure $ ( Q \mathbin{ _\mu \times _\mu } Q ) / G \rightrightarrows Q / G $, and the quotient map $ Q \mathbin{ _\mu \times _\mu } Q \rightarrow ( Q \mathbin{ _\mu \times _\mu } Q ) / G $ is a morphism of Lie groupoids over $ Q \rightarrow Q / G $.
\end{Lemma}

\begin{proof} The fact that $ \bigl( g, ( q _1 , q _0 )\bigr) \mapsto ( g q _1 , g q _0 ) $ is a free and proper groupoid action follows immediately from the fact that, by assumption, $ ( g, q) \mapsto g q $ is. As stated in Remark~\ref{rmk:freeProperQuotient}, the freeness and properness of these actions imply that $ Q / G $ and $ (Q \mathbin{ _\mu \times _\mu } Q)/G $ are smooth manifolds, so it suf\/f\/ices to specify the groupoid structure maps for $ (Q \mathbin{ _\mu \times _\mu } Q)/G \rightrightarrows Q / G $. These may be taken to be formally identical to those for the gauge groupoid in Example~\ref{ex:gaugeGroupoid}, i.e.,
 \begin{gather*}
 \alpha \bigl( [ q _1 , q _0 ] \bigr) = [q _0] , \qquad \beta \bigl( [ q _1, q _0 ] \bigr) = [q _1] , \qquad m \bigl( [q _2 , q
 _1 ] , [ q _1 , q _0 ] \bigr) = [ q _2 , q _0 ] ,\\
 \epsilon \bigl( [q] \bigr) = [ q, q ] , \qquad i \bigl( [ q _1 , q _0] \bigr) = [q _0, q _1].
 \end{gather*}
As with the gauge groupoid, it is simple to check directly that these satisfy the conditions of Def\/inition~\ref{def:groupoid}, so this is a~Lie groupoid. Finally, using $ \widetilde{ \alpha } $, $ \widetilde{ \beta } , \ldots $ to denote the structure maps on $ Q \mathbin{ _\mu \times _\mu } Q \rightrightarrows Q $, we observe that
 \begin{gather*}
\alpha \bigl( [ q _1 , q _0 ] \bigr) = \bigl[ \widetilde{ \alpha } ( q _1, q _0 ) \bigr] , \qquad \beta \bigl( [ q _1, q _0 ] \bigr) = \bigl[
\widetilde{ \beta } ( q _1, q _0 ) \bigr] , \\
 m \bigl( [q _2 , q _1 ] , [ q _1 , q _0 ] \bigr) = \bigl[ \widetilde{ m } \bigl((q _2, q _1), (q_1, q _0 ) \bigr) \bigr] ,\\
\epsilon \bigl( [q] \bigr) = \bigl[ \widetilde{ \epsilon } (q) \bigr] , \qquad i \bigl( [ q _1 , q _0] \bigr) = \bigl[\widetilde{ \imath } (q _1, q _0 ) \bigr] ,
 \end{gather*}
so the quotient map preserves the structure maps and hence is a Lie groupoid morphism.
\end{proof}

\begin{Lemma} \label{lem:freeProper}
The action of a Lie groupoid $ G \rightrightarrows M $ on $Q \rightarrow M$ is free $($resp., proper$)$ if and only if the induced action on $ V Q \rightarrow M $ is free $($resp., proper$)$.
\end{Lemma}

\begin{proof} If $G$ acts freely on $Q$, then $ g _\ast v = v $ implies $ g \bigl( \tau (v) \bigr) = \tau (v) $, so $ g = \epsilon \bigl( \mu \bigl( \tau (v) \bigr) \bigr)= \epsilon \bigl( (\mu \circ \tau) (v) \bigr) $, and hence $G$ acts freely on $ V Q $. Conversely, if $G$ acts freely on $ V Q $, then $ g q = q $ implies $ g _\ast 0 _q = 0 _q $ so $ g = \epsilon \bigl( ( \mu \circ \tau ) ( 0 _q ) \bigr) = \epsilon \bigl( \mu (q) \bigr) $, and hence $G$ acts freely on~$Q$.

The proof of properness essentially amounts to chasing compact sets around the following diagram:
 \begin{equation*}
 \begin{tikzcd}
 G \mathbin{ _\alpha \times _{ \mu \circ \tau } } V Q \arrow[r]
 \arrow[d, shift left=1ex, "\mathrm{id} \times \tau"] & VQ
 \mathbin{ _{\mu \circ \tau } \times _{ \mu \circ
 \tau } } V Q \arrow[d, shift left=1ex, "\tau \times \tau"] \\
 G \mathbin{ _\alpha \times _\mu } Q \arrow[r] \arrow[u, shift
 left=1ex, "\mathrm{id} \times 0 "] & Q \mathbin{ _\mu \times
 _\mu } Q \arrow[u, shift left=1ex, "0 \times 0"] .
 \end{tikzcd}
 \end{equation*}
 First, suppose $G$ acts properly on $Q$. If $ K \subset V Q \mathbin{ _{\mu \circ \tau } \times _{\mu \circ \tau}} V Q $ is compact, then we wish to show that the preimage,
 \begin{gather*}
 \bigl\{ ( g, v ) \in G \mathbin{ _\alpha \times _{ \mu \circ \tau } } V Q \colon ( v, g _\ast v) \in K \bigr\} ,
 \end{gather*}
is also compact. Observe that $ \bigl\{ v \in V Q \colon ( v, g _\ast v) \in K \bigr\} $ is compact by the continuity of $ ( v, g _\ast v ) \mapsto v $, and $ \bigl\{ g \in G \colon ( v, g _\ast v) \in K \bigr\} $ is compact by the continuity of $ ( v, g _\ast v ) \mapsto ( q, g q ) $, with $ q = \tau (v) $, the properness of $ ( g, q ) \mapsto ( q, gq ) $, and the continuity of $ ( g, q ) \mapsto g $. Hence, the preimage in question is also compact, so $G$ acts properly on $ V Q $.

 Conversely, suppose $G$ acts properly on $ V Q $. If $ K \subset Q \mathbin { _\mu \times _\mu } Q $ is compact, then so is $ \bigl\{ ( 0 _q , g _\ast 0 _q ) \in V Q \mathbin{ _{\mu \circ \tau } \times _{\mu \circ \tau}} V Q \colon ( q, g q ) \in K \bigr\} $, and by properness, so is $ \bigl\{ ( g, 0 _q ) \in G \mathbin{ _\alpha \times _{ \mu \circ \tau } } V Q \colon ( q, gq ) \in K \bigr\} $. Finally, the preimage,
 \begin{gather*}
 \bigl\{ ( g, q ) \in G \mathbin { _\alpha \times _\mu } Q \colon ( g, q ) \in K \bigr\},
 \end{gather*}
 is compact by the continuity of $ ( g, 0 _q ) \mapsto ( g, q ) $, so $G$ acts properly on $Q$.
\end{proof}

\subsection{Lie algebroid of a Lie groupoid}

Before discussing reduction by an arbitrary free and proper groupoid action, we f\/irst consider the important special case where a groupoid
acts on itself by left multiplication. (This can be thought of as the ``groupoid version'' of Euler--Poincar\'e reduction, which is the special case of Lagrange--Poincar\'e reduction where $ Q = G $ is a~Lie group.)

Recall from Example~\ref{ex:multiplicationAction} that a Lie groupoid $ G \rightrightarrows M $ acts freely and properly on itself (as the f\/ibered manifold $ \beta \colon G \rightarrow M $) by left multiplication. Lemma~\ref{lem:freeProper} implies that this induces a~free and proper action of $G$ on the $\beta$-vertical bundle $ V ^\beta G \rightarrow M $. (Since $G$ can be seen as a~f\/ibered manifold in two dif\/ferent ways, $ \alpha \colon G \rightarrow M $ and $ \beta \colon G \rightarrow M $, we denote the corresponding vertical bundles by $ V ^\alpha G $ and $ V ^\beta G $ to avoid any possible confusion.) Since the orbit of $ v \in V ^\beta _g G $ is uniquely determined by its representative at the identity section, $ ( g ^{-1} ) _\ast v \in V ^\beta _{ \epsilon ( \alpha (g) ) } G $, we can identify the quotient $ V ^\beta G / G $ with the vector bundle $ A G = V ^\beta _{ \epsilon (M) } G $ over $M$.

This vector bundle $ A G \rightarrow M $ is in fact a Lie algebroid, called the \emph{Lie algebroid of $G$}. The anchor map is given by the restriction of $ \alpha _\ast \colon T G \rightarrow T M $ to $ A G $. Furthermore, the identif\/ication of $ A G $ with $ V ^\beta G / G $ implies that sections $ X \in \Gamma ( A G ) $ correspond to $G$-invariant, $\beta$-vertical vector f\/ields $ \overleftarrow{ X } \in \mathfrak{X} _\beta (G) $, with $ \overleftarrow{ X } (g) = g _\ast X \bigl( \alpha (g) \bigr) $. The bracket $ [ X , Y ] $ of $ X, Y \in \Gamma ( A G ) $ is then def\/ined so that $ \overleftarrow{ [X,Y] } = [ \overleftarrow{ X }, \overleftarrow{ Y }] $, where the bracket on the right-hand side of this expression is just the Jacobi--Lie bracket of vector f\/ields on~$G$. (See Mackenzie~\cite{Mackenzie2005}.)

\begin{Example}Let $G$ be a Lie group, so that $ G \rightrightarrows \bullet $ is a Lie groupoid. Since $ \beta $ is trivial, we have $ V ^\beta G = T G $, and hence $ A G = T _e G = \mathfrak{g} \rightarrow \bullet $, where $ \mathfrak{g} $ is the Lie algebra of $G$ and $ e = \epsilon ( \bullet ) \in G $ is the identity element of $G$.
\end{Example}

\begin{Example} \label{ex:verticalAlgebroid} For the pair groupoid $ Q \times Q \rightrightarrows Q $, we have $ V ^\beta ( Q \times Q ) = T Q \times Q $, and hence $ A ( Q \times Q ) = T Q \mathbin{ _\tau \times } Q \cong T Q \rightarrow Q $.

More generally, if we consider the groupoid $ Q \mathbin{ _\mu \times _\mu } Q \rightrightarrows Q $ for a f\/ibered manifold $ Q \rightarrow M $, then $ V ^\beta ( Q \mathbin{ _\mu \times _\mu } Q ) = V Q \mathbin{ _{\mu \circ \tau } \times _\mu } Q $, and hence $ A ( Q \mathbin{ _\mu \times _\mu } Q ) = V Q \mathbin{ _\tau \times } Q \cong V Q \rightarrow Q $.
\end{Example}

\begin{Example} \label{ex:gaugeAlgebroid}
For the gauge groupoid $ ( Q \times Q ) / G \rightrightarrows Q / G $ of a principal bundle $ Q \rightarrow Q / G $, we have $ V ^\beta \bigl( ( Q \times Q ) / G \bigr) = (T Q \times Q) / G $, and hence $ A \bigl( ( Q \times Q ) / G \bigr) = ( T Q \mathbin{ _\tau \times } Q ) / G \cong T Q / G \rightarrow Q / G $. This is called the \emph{gauge algebroid} (or \emph{Atiyah algebroid}) of the principal bundle.

More generally, considering the groupoid $ (Q \mathbin{ _\mu \times _\mu } Q) / G \rightrightarrows G $ of a~principal $G$-space, we have $ V ^\beta \bigl( (Q \mathbin{ _\mu \times _\mu } Q) / G \bigr) = ( V Q \mathbin{ _{ \mu \circ \tau } \times _\mu } Q ) / G $, and hence $ A \bigl( (Q \mathbin{ _\mu \times _\mu } Q) / G \bigr) = ( V Q \mathbin{ _\tau \times } Q ) / G \cong V Q / G \rightarrow Q / G $.
\end{Example}

\begin{Remark} The relationship between a groupoid $G$ and its algebroid $ A G $ has an interesting application to the \emph{discretization} of Lagrangian mechanics, which can be used to develop structure-preserving numerical integrators. In this approach, pioneered by Weinstein~\cite{Weinstein1996} (see also Marrero et al.~\cite{MaMaMa2006,MaMaSt2015}, Stern~\cite{Stern2010}), one replaces the Lagrangian $ L \colon A G \rightarrow \mathbb{R} $ by a \emph{discrete Lagrangian} $ L _h \colon G \rightarrow \mathbb{R} $, replaces $AG$-paths by sequences of composable arrows in $G$, and uses a variational principle to derive discrete equations of motion. In particular, using $ G = Q \times Q \rightrightarrows Q $
to discretize $ A G = T Q \rightarrow Q $ gives the framework of \emph{variational integrators} (cf.\ Moser and Veselov~\cite{MoVe1991}, Marsden and West~\cite{MaWe2001}).
\end{Remark}

\subsection{Reduction by a groupoid action}

Recall from Lemma~\ref{lem:freeProper} that if $ G \rightrightarrows M $ acts freely and properly on $ Q \rightarrow M $, then it also acts freely and properly on $ V Q \rightarrow M $. In other words, $ V Q $ is also a principal $G$-space, equipped with a quotient map $ V Q \rightarrow V Q / G $. We have seen that $ V Q $ is also a Lie algebroid, and moreover, in Example~\ref{ex:verticalAlgebroid}, that it is the Lie algebroid of the Lie groupoid $ Q \mathbin{ _\mu \times _\mu } Q \rightrightarrows Q $. Similarly, from Example~\ref{ex:gaugeAlgebroid}, we have that $ V Q / G $ is the Lie algebroid of the Lie groupoid $ (Q \mathbin{ _\mu \times _\mu } Q)/G \rightrightarrows Q/G $.

Therefore, in order to perform reduction using Theorem~\ref{thm:weinsteinMartinez}, it suf\/f\/ices to show that the quotient map $ V Q \rightarrow V Q / G $ is in fact a Lie algebroid morphism.

\begin{Lemma} \label{lem:quotientMorphism}
Let $ G \rightrightarrows M $ be a Lie groupoid and $ Q \rightarrow M $ a principal $G$-space. Then the quotient map $ V Q \rightarrow V Q / G $ is a Lie algebroid morphism covering $ Q \rightarrow Q / G $.
\end{Lemma}

\begin{proof}We can use a result stated in Mackenzie~\cite[Proposition~4.3.4]{Mackenzie2005}, which says that a~morphism of Lie groupoids $ G \rightarrow G ^\prime $ induces a corresponding morphism of Lie algebroids $ A G \rightarrow A G ^\prime $. This def\/ines the so-called \emph{Lie functor} between the categories of Lie groupoids and Lie algebroids, taking objects $ G \mapsto A G $ and morphisms $ ( G \rightarrow G ^\prime ) \mapsto ( A G \rightarrow A G
 ^\prime ) $.

Now, we have already proved in Lemma~\ref{lem:quotientGroupoid} that the quotient map $ Q \mathbin{ _\mu \times _\mu } Q \rightarrow ( Q \mathbin{ _\mu \times _\mu } Q ) / G $ is a morphism of Lie groupoids, so applying the Lie functor to this morphism proves the result.
\end{proof}

\begin{Theorem} \label{thm:reduction}
Let $ G \rightrightarrows M $ be a Lie groupoid and $ Q \rightarrow M $ a principal $G$-space. Suppose the Lagrangian $ L \colon V Q \rightarrow \mathbb{R} $ is $G$-invariant, i.e., that it factors through the quotient morphism $ \Phi \colon V Q \rightarrow V Q / G $ as $ L = \ell \circ \Phi $, where $ \ell \colon V Q / G \rightarrow \mathbb{R} $ is called the reduced Lagrangian. Then $a \in \mathcal{P} _V (Q) $ is a solution path for $L$ if and only if $ \Phi \circ a \in \mathcal{P} _\rho ( V Q / G ) $ is a solution path for~$ \ell $.
\end{Theorem}

\begin{proof}
Apply Theorem~\ref{thm:weinsteinMartinez} to the (f\/iberwise-surjective) Lie algebroid morphism def\/ined in Lemma~\ref{lem:quotientMorphism}.
\end{proof}

\begin{Example} \label{ex:groupActionReduction}
When $G \rightrightarrows \bullet $ is a Lie group acting freely and properly on $ Q \rightarrow \bullet $, Theorem~\ref{thm:reduction} corresponds to ordinary Lagrangian reduction from $ T Q $ to $ T Q / G $, yielding the Lagrange--Poincar\'e equations of Section~\ref{sec:lagrangePoincare}. In the special case where $ Q = G $ acts on itself by multiplication, this gives Euler--Poincar\'e reduction from $ T G $ to $ T G / G \cong \mathfrak{g} $.
\end{Example}

\begin{Example}\label{ex:groupoidSelfReduction}
Suppose $G \rightrightarrows M $ is a Lie groupoid acting on itself by multiplication, so that the quotient morphism is $ \Phi \colon V ^\beta G \rightarrow V ^\beta G / G = A G $. If $ L \colon V ^\beta G \rightarrow \mathbb{R} $ and $ \ell \colon A G \rightarrow \mathbb{R} $ are Lagrangians satisfying $ L = \ell \circ \Phi $, then Theorem~\ref{thm:reduction} implies that the vertical Euler--Lagrange equations (Section~\ref{sec:EL}) on $ V ^\beta G $ reduce to the Euler--Lagrange--Poincar\'e equations (Section~\ref{sec:algebroids}) for the Lie algebroid $ A G $. (This special case appears in Weinstein~\cite[Theorem~5.3]{Weinstein1996}.) The even more special case where $ G \rightrightarrows \bullet $ is a Lie group again gives Euler--Poincar\'e reduction on the Lie algebra $ \mathfrak{g} $.
\end{Example}

\section{The Hamilton--Pontryagin principle and reduction}\label{sec:HP}

In this section, we extend the foregoing theory to the Hamilton--Pontryagin variational principle introduced by Yoshimura and Marsden~\cite{YoMa2006b} as a~generalization of Hamilton's variational principle. This principle is especially useful for the study of ``implicit Lagrangian systems'' that arise in mechanical and control systems with nonholonomic or Dirac constraints. (See also Yoshimura and Marsden~\cite{YoMa2006a} for the non-variational approach to such systems, as well as Yoshimura and Marsden~\cite{YoMa2007} for the associated reduction theory.)

We begin, in Section~\ref{sec:hpReview}, with a brief review of the Hamilton--Pontryagin principle for ordinary manifolds. We then generalize it, in Section~\ref{sec:hpFibered}, to f\/ibered manifolds and their (co)vertical bundles, as we did for Hamilton's principle in Section~\ref{sec:EL}. In Section~\ref{sec:hpAlgebroids}, we generalize the Hamilton--Pontryagin principle even further to mechanics on Lie algebroids and their duals, as was done for Hamilton's principle in Section~\ref{sec:algebroids}. Finally, in Section~\ref{sec:hpReduction}, we discuss reduction of the Hamilton--Pontryagin principle by Lie algebroid morphisms, as in the Weinstein--Mart\'inez reduction theorem (Theorem~\ref{thm:weinsteinMartinez}), and apply this to the special case of groupoid symmetries for a f\/ibered manifold, as in Theorem~\ref{thm:reduction}.

\subsection{Hamilton--Pontryagin principle for ordinary manifolds}\label{sec:hpReview}

We begin with a quick review of the Hamilton--Pontryagin principle for ordinary (non-f\/ibered) manifolds, as introduced in Yoshimura and Marsden~\cite{YoMa2006b}.

Let $ L \colon T Q \rightarrow \mathbb{R} $ be a Lagrangian. The \emph{Hamilton--Pontryagin action} is the functional $ S \colon \mathcal{P} ( T Q \oplus T ^\ast Q ) \rightarrow \mathbb{R} $ def\/ined, in f\/iber coordinates, by
\begin{gather*}
 S ( q, v , p ) = \int _0 ^1 \bigl( L \bigl( q (t) , v (t) \bigr) + \bigl\langle p (t) , \dot{q} (t) - v (t) \bigr\rangle \bigr) \,\mathrm{d}t .
\end{gather*}
Here, $ ( q, v, p ) $ is an arbitrary path in the \emph{Pontryagin bundle} $ T Q \oplus T ^\ast Q $. We emphasize that no restrictions are placed on this path~-- in particular, the second-order curve condition $ \dot{q} = v $ is not \emph{a priori} required.

The path $ ( q, v, p ) $ satisf\/ies the \emph{Hamilton--Pontryagin principle} if $ \mathrm{d} S ( \delta q, \delta v, \delta p ) = 0 $ for all variations $ ( \delta q, \delta v, \delta p ) \in T _{ ( q, v, p ) } \mathcal{P} ( T Q \oplus T ^\ast Q ) $ such that $ \delta q (0) = 0 $ and $ \delta q (1) = 0 $. (That is, the endpoints of $q$ are f\/ixed, while the endpoints of $v$ and $p$ are unrestricted.) In local coordinates, we have
\begin{gather*}
 \mathrm{d} S ( \delta q, \delta v , \delta p ) = \int _0 ^1 \left( \frac{ \partial L }{ \partial q ^i } ( q, v ) \delta q ^i + \frac{ \partial L }{ \partial v ^i } ( q, v ) \delta v ^i + p _i ( \delta \dot{q} ^i - \delta v ^i ) + \delta p _i ( \dot{q} ^i - v ^i ) \right) \mathrm{d}t \\
\hphantom{\mathrm{d} S ( \delta q, \delta v , \delta p )}{} = \int _0 ^1 \left[ \left( \frac{ \partial L }{ \partial q ^i } (
 q, v ) - \dot{p} _i \right) \delta q ^i + \left( \frac{ \partial L }{ \partial v ^i } ( q, v ) - p _i \right) \delta v ^i + \delta p _i
 ( \dot{q} ^i - v ^i ) \right] \mathrm{d}t .
\end{gather*}
Hence, this vanishes when $ ( q, v, p ) $ satisf\/ies the dif\/ferential-algebraic equations
\begin{gather*}
 \frac{ \partial L }{ \partial q ^i } ( q, v ) - \dot{p} _i = 0 , \qquad \frac{ \partial L }{ \partial v ^i } ( q, v ) - p _i = 0 , \qquad \dot{q} ^i - v ^i = 0 ,
\end{gather*}
which Yoshimura and Marsden~\cite{YoMa2006b} call the \emph{implicit Euler--Lagrange equations}. The three systems of equations correspond, respectively,
to the Euler--Lagrange equations, the Legendre transform, and the second-order curve condition. (Note that the conjugate momentum~$ p $ acts like a ``Lagrange multiplier'' enforcing the second-order curve condition $ \dot{q} = v $.)

In this sense, the Hamilton--Pontryagin approach generalizes and unif\/ies the symplectic and variational approaches to Lagrangian mechanics.

\subsection{Hamilton--Pontryagin for f\/ibered manifolds}\label{sec:hpFibered}

Suppose, more generally, that $ L \colon V Q \rightarrow \mathbb{R} $ is a Lagrangian on the vertical bundle of a f\/ibered manifold $ Q \rightarrow M $. Recall that $ V Q $ and $ V ^\ast Q $ can both be viewed as f\/ibered manifolds over $M$, and thus so can $ V Q \oplus V ^\ast Q $, which we call the \emph{vertical Pontryagin bundle}. It follows that we may def\/ine a Banach manifold of vertical paths $ \mathcal{P} _V ( V Q \oplus V ^\ast Q ) $ and its bundle of vertical variations $ V \mathcal{P} _V ( V Q \oplus V ^\ast Q ) $.

\begin{Definition}
Given a Lagrangian $ L \colon V Q \rightarrow \mathbb{R} $, the \emph{Hamilton--Pontryagin action} $ S \colon \mathcal{P} _V ( V Q \oplus V ^\ast Q ) \rightarrow \mathbb{R} $ is def\/ined, in f\/iber coordinates, by
 \begin{gather*}
 S (q, v, p) = \int _0 ^1 \bigl( L \bigl( q (t) , v (t) \bigr) + \bigl\langle p (t) , \dot{q} (t) - v (t) \bigr\rangle \bigr) \,\mathrm{d}t .
 \end{gather*}
 A vertical path $ ( q, v , p ) \in \mathcal{P} _V ( V Q \oplus V ^\ast Q ) $ is said to satisfy the \emph{Hamilton--Pontryagin principle} if $ \mathrm{d} S ( \delta q, \delta v , \delta p ) = 0 $ for all vertical variations $ ( \delta q, \delta v , \delta p ) \in V _{ ( q, v, p ) } \mathcal{P} _V ( V Q \oplus V ^\ast Q ) $ such that $ \delta q (0) = 0 $ and $ \delta q (1) = 0 $.
\end{Definition}

\begin{Theorem} A vertical path $ (q,v,p) \in \mathcal{P} _V ( V Q \oplus V ^\ast Q ) $ satisfies the Hamilton--Pontryagin principle if and only if, in fiber-adapted local coordinates $ q = ( x ^\sigma , y ^i ) $, it satisfies the \emph{implicit vertical Euler--Lagrange equations},
 \begin{gather} \label{eq:implicitVEL}
 \dot{x} ^\sigma = 0 , \qquad \dot{p} _i = \frac{ \partial L }{ \partial y ^i } ( q, v ) , \qquad p _i = \frac{ \partial L }{ \partial v ^i } ( q, v ) , \qquad \dot{y} ^i = v ^i .
\end{gather}
\end{Theorem}

\begin{proof} The equations $ \dot{x} ^\sigma = 0 $ are simply the vertical path condition. Given a vertical variation $ (\delta q, \delta v, \delta p ) \in V _{(q,v,p)} \mathcal{P} _V (V Q \oplus V ^\ast Q) $ satisfying $ \delta q (0) = 0 $ and $ \delta q (1) = 0 $,
 \begin{gather*}
 \mathrm{d} S ( \delta q, \delta v , \delta p ) = \int _0 ^1 \left( \frac{ \partial L }{ \partial y ^i } ( q, v ) \delta y ^i + \frac{ \partial L }{ \partial v ^i } ( q, v ) \delta v ^i + p _i ( \delta \dot{y} ^i - \delta v ^i ) + \delta p _i ( \dot{y} ^i - v ^i ) \right) \mathrm{d}t \\
\hphantom{\mathrm{d} S ( \delta q, \delta v , \delta p )}{} = \int _0 ^1 \left[ \left( \frac{ \partial
 L }{ \partial y ^i } ( q, v ) - \dot{p} _i \right) \delta y ^i + \left( \frac{ \partial L }{ \partial v ^i } ( q, v ) - p _i \right) \delta v ^i + \delta p _i ( \dot{y} ^i - v ^i ) \right] \mathrm{d}t .
 \end{gather*}
This vanishes for arbitrary $ ( \delta q, \delta v , \delta p ) $ if and only if each of the components in the integrand vanishes, which completes the proof.
\end{proof}

\subsection{Hamilton--Pontryagin for arbitrary Lie algebroids}\label{sec:hpAlgebroids}

We next generalize the Hamilton--Pontryagin principle to a Lagrangian $ L \colon A \rightarrow \mathbb{R} $, where $A \rightarrow Q $ is an arbitrary Lie algebroid. The previous subsections will then correspond to the special cases $ A = T Q $ and $ A = V Q $, respectively.

One might expect that the appropriate generalization of paths in $ T Q \oplus T ^\ast Q $ or $ V Q \oplus V ^\ast Q $ would be paths in $ A \oplus A ^\ast $. However, these generally do not contain suf\/f\/icient information to recover the $A$-path condition (the generalization of the second-order curve condition). Instead, we consider an alternative class of paths that we call $ ( A, A ^\ast ) $-paths.

\begin{Definition} An \emph{$ ( A, A ^\ast ) $-path} consists of the following components:
\begin{enumerate}[label=(\roman*)]\itemsep=0pt
\item an $A$-path $ a \in \mathcal{P} _\rho (A) $ over some base path $ q \in \mathcal{P} (Q) $;
\item a path $ v \in \mathcal{P} (A) $, not necessarily an $A$-path, over $q$;
\item a path $ p \in \mathcal{P} ( A ^\ast ) $ over $q$.
\end{enumerate}
We denote this by $ ( a, v, p ) \in \mathcal{P} ( A, A ^\ast ) $.
\end{Definition}

\begin{Example}
Any path $ ( q, v, p ) \in \mathcal{P} ( T Q \oplus T ^\ast Q ) $ can be identif\/ied with the $ ( T Q, T ^\ast Q ) $-path $ ( \dot{q}, v, p ) \in \mathcal{P} ( T Q , T ^\ast Q ) $. More generally, $ ( q, v, p ) \in \mathcal{P} _V ( V Q \oplus V ^\ast Q ) $ can be identif\/ied with $ ( \dot{q}, v, p ) \in \mathcal{P} ( V Q , V ^\ast Q ) $. Thus, $ \mathcal{P} ( VQ , V ^\ast Q ) \cong \mathcal{P} _V ( V Q \oplus V ^\ast Q ) $.

In this special case, the base path has a unique $A$-path prolongation, so it suf\/f\/ices to consider paths in $ A \oplus A ^\ast $~-- but this is not the case in general.
\end{Example}

\begin{Example} \label{ex:lieAlgebraPartI}
Let $ \mathfrak{g} $ be a Lie algebra. Since all paths in $\mathfrak{g} \rightarrow \bullet $ are $\mathfrak{g}$-paths, it follows that a~$ ( \mathfrak{g} , \mathfrak{g} ^\ast ) $ path $ ( a, v , p ) \in \mathcal{P} ( \mathfrak{g} , \mathfrak{g} ^\ast ) $ consists of two (generally distinct) paths $ a, v \in \mathcal{P} ( \mathfrak{g} ) $ and a path $ p \in \mathcal{P} ( \mathfrak{g} ^\ast ) $. Thus, $ \mathcal{P} ( \mathfrak{g} , \mathfrak{g} ^\ast ) \cong \mathcal{P} ( \mathfrak{g} \oplus \mathfrak{g} \oplus \mathfrak{g} ^\ast ) $.
\end{Example}

\begin{Definition}
An \emph{admissible variation} of $ ( a, v, p ) \in \mathcal{P} ( A, A ^\ast ) $ consists of an admissible variation $ X _{ b, a } \in \mathcal{F} _a (A) $ of the $A$-path $a$, together with arbitrary variations $ \delta v \in T _v \mathcal{P} ( A ) $ and $ \delta p \in T _p \mathcal{P} ( A ^\ast ) $, such that all agree on the horizontal component $ \delta q = \rho (b) \in \mathcal{P} _q (Q) $. That is, if $ \tau \colon A \rightarrow Q $ and $ \pi \colon A ^\ast \rightarrow Q $ are the bundle projections, we require $ \tau _\ast (v) = \pi _\ast (p) = \rho (b) $. Following Remark~\ref{rmk:admissibleSubbundle}, we denote this subbundle of admissible variations by $ \mathcal{F} (A, A ^\ast ) \subset T \mathcal{P} ( A, A ^\ast ) $.
\end{Definition}

\begin{Remark}
Given a $ T Q $-connection $ \nabla $, the admissible variation $ ( X _{ b, a } , \delta v, \delta p ) \in \mathcal{F} _{ ( a, v, p ) } ( A , A ^\ast ) $ has components $ X _{ b, a } ^\text{ver} = \overline{ \nabla } _a b $ and $ X _{ b, a } ^\text{hor} = \delta v ^\text{hor} = \delta p ^\text{hor} = \rho (b) $, while $ \delta v ^\text{ver} $ and $ \delta p ^\text{ver} $ are arbitrary paths in $A$ and $ A ^\ast $, respectively.
\end{Remark}

\begin{Example} Continuing Example~\ref{ex:lieAlgebraPartI}, let us consider the case where $ \mathfrak{g} $ is a Lie algebra. Given a~$ ( \mathfrak{g} , \mathfrak{g} ^\ast ) $ path $ ( a, v, p ) $, we identify $a$ with its time-dependent section $ \xi (t) = \xi ( t, \bullet ) = a (t) $. Then an admissible
 variation of $ ( a, v, p ) $ has the form $ ( \dot{ \xi } + [ \xi, \eta ] , \delta v , \delta p ) $, where $ \eta $, $ \delta v $, and $ \delta p $ are arbitrary paths in $ \mathfrak{g} $.

Equivalently, assuming $\mathfrak{g}$ is the Lie algebra of a Lie group $G$, let $ g \in \mathcal{P} (G) $ be a path integrating $\xi$, i.e., $ g (t) = g (0) \exp( t \xi ) $, so that $ \xi = ( g ^{-1} ) _\ast \dot{g} $. It follows that arbitrary variations $ ( \delta g , \delta v, \delta p ) \in T _{ (g,v,p) } \in \mathcal{P} ( G \times \mathfrak{g} \times \mathfrak{g} ^\ast ) $ correspond precisely to admissible variations $ ( \dot{ \xi } + [ \xi, \eta ] , \delta v , \delta p ) $ of $ ( \xi, v, p ) \in \mathcal{P} ( \mathfrak{g} , \mathfrak{g} ^\ast ) $, where $ \eta = ( g ^{-1} ) _\ast \delta g $. This special case corresponds to the approach of Yoshimura and Marsden~\cite{YoMa2007} and Bou-Rabee and Marsden~\cite{BoMa2009} for Hamilton--Pontryagin mechanics on Lie algebras, where one considers paths in $ \mathcal{P} ( G \times \mathfrak{g} \times \mathfrak{g} ^\ast ) $ and then left-reduces by $G$.
\end{Example}

\begin{Definition}
Given a Lagrangian $ L \colon A \rightarrow \mathbb{R} $, the \emph{Hamilton--Pontryagin action} $ S \colon \mathcal{P} ( A , A ^\ast ) \rightarrow \mathbb{R} $ is def\/ined by
 \begin{gather*}
 S ( a, v , p ) = \int _0 ^1 \bigl( L \bigl( v (t) \bigr) + \bigl\langle p (t) , a (t) - v (t) \bigr\rangle \bigr) \,\mathrm{d}t ,
 \end{gather*}
and $ (a,v,p) \in \mathcal{P} ( A , A ^\ast ) $ is said to satisfy the \emph{Hamilton--Pontryagin principle} if $ \mathrm{d} S ( X _{ b, a } , \delta v , \delta p ) = 0 $ for all admissible variations $ ( X _{ b, a }, \delta v , \delta p ) \in \mathcal{F} _{ (a,v,p) } ( A, A ^\ast ) $.
\end{Definition}

\begin{Theorem}An $ ( A, A ^\ast ) $-path $ ( a, v , p ) \in \mathcal{P} ( A, A ^\ast ) $ satisfies the Hamilton--Pontryagin principle if and only if, given a~$ T Q $-connection $ \nabla $ on $A$, it satisfies the differential-algebraic equations,
 \begin{gather} \label{eqn:implicitELP}
 \rho ^\ast \mathrm{d} L ^\textup{hor} (v) + \overline{ \nabla } _a ^\ast p = 0 , \qquad \mathrm{d} L ^\textup{ver} (v) - p = 0 , \qquad a - v = 0 .
 \end{gather}
\end{Theorem}

\begin{proof} Given $ ( X _{ b, a } , \delta v, \delta p ) \in \mathcal{F} _{ (a,v,p) } ( A , A ^\ast ) $, we compute
 \begin{gather*}
 \mathrm{d} S ( X _{ b, a } , \delta v , \delta p ) = \int _0 ^1 \Bigl( \bigl\langle \mathrm{d} L ^\text{hor} (v) ,
 \rho (b) \bigr\rangle + \bigl\langle \mathrm{d} L ^\text{ver} (v), \delta v ^\text{ver} \bigr\rangle \\
 \hphantom{\mathrm{d} S ( X _{ b, a } , \delta v , \delta p ) = \int _0 ^1}{} + \langle p , \overline{ \nabla } _a
 b - \delta v ^\text{ver} \rangle + \langle \delta p^\text{ver} , a - v \rangle \Bigr) \mathrm{d}t\\
\hphantom{\mathrm{d} S ( X _{ b, a } , \delta v , \delta p )}{}
 = \int _0 ^1 \Bigl( \bigl\langle \rho ^\ast \mathrm{d} L ^\text{hor} (v) + \overline{ \nabla } _a ^\ast p , b
 \bigr\rangle + \bigl\langle \mathrm{d} L ^\text{ver} (v) - p , \delta v ^\text{ver} \bigr\rangle + \langle \delta p ^\text{ver}, a - v \rangle \Bigr)
\mathrm{d}t .
 \end{gather*}
 The Hamilton--Pontryagin principle is satisf\/ied if and only if each term in the integrand vanishes, and since $b$, $ \delta v ^\text{ver} $, and $ \delta p ^\text{ver} $ are arbitrary, the result follows.
\end{proof}

We call the dif\/ferential-algebraic equations \eqref{eqn:implicitELP} the \emph{implicit Euler--Lagrange--Poincar\'e equations}. As we did in Theorem~\ref{thm:coordinateELP} we can give an equivalent expression for~\eqref{eqn:implicitELP} in local coordinates.

\begin{Theorem}
Let $ q ^i $ be local coordinates for $Q$, $ \{ e _I \} $ be a local basis of sections of $A$, $ \{ e ^I \} $ be the dual basis of local sections of $ A ^\ast $, $ \nabla $ be the locally trivial $ T Q $-connection, and $ \rho ^i _I $ and $ C _{ I J } ^K $ be the local-coordinate representations of $ \rho $ and $ [\cdot , \cdot ] $. Let $ ( a, v, p ) \in \mathcal{P} ( A \oplus A \oplus A ^\ast ) $ have the local-coordinate representations $ a (t) = \xi ^I (t) e _I \bigl( q (t) \bigr) $, $ v (t) = v ^I (t) e _I \bigl( q (t) \bigr) $, and $ p (t) = p _I (t) e ^I \bigl( q (t) \bigr) $. Then $ ( a, v, p ) \in \mathcal{P} ( A , A ^\ast ) $ if and only if $ \dot{q} ^i = \rho ^i _I \xi ^I $, and $ ( a, v, p ) $ satisfies the implicit Euler--Lagrange--Poincar\'e equations~\eqref{eqn:implicitELP} if and only if
\begin{gather*}
\rho ^i _I \frac{ \partial L }{ \partial q ^i } - C ^K _{ I J } \xi ^J p _K - \dot{p} _I = 0 , \qquad
\frac{ \partial L }{ \partial \xi ^I } - p _I = 0 , \qquad \xi ^I - v ^I = 0 .
\end{gather*}
\end{Theorem}

\begin{proof} The proof is a straightforward computation, following Theorem~\ref{thm:coordinateELP}.
\end{proof}

\subsection{Reduction by groupoid symmetries}\label{sec:hpReduction}

Finally, we consider the reduction of Hamilton--Pontryagin mechanics by a Lie algebroid morphism $ \Phi \colon A \rightarrow A ^\prime $, as in Theorem~\ref{thm:weinsteinMartinez}. Here, though, we will require the slightly stronger assumption that $ \Phi $ be a f\/iberwise isomorphism. (This was actually assumed in the original Lie algebroid reduction theorem of Weinstein~\cite{Weinstein1996}, although Mart\'{\i}nez~\cite{Martinez2008} showed that it could be relaxed.) This stronger assumption is needed since $ \Phi ^\ast \colon A ^{ \prime \ast } \rightarrow A ^\ast $ points in the ``wrong direction'' for reduction from $ ( A, A ^\ast ) $ to $ ( A ^\prime , A ^{ \prime \ast } ) $, so we need f\/iberwise invertibility to map $ A ^\ast \rightarrow A ^{ \prime \ast }$.

\begin{Theorem} \label{thm:hpMorphismReduction}Let $ \Phi \colon A \rightarrow A ^\prime $ be a morphism of Lie algebroids, and suppose $ L \colon A \rightarrow \mathbb{R} $ and $ L ^\prime \colon A ^\prime \rightarrow \mathbb{R} $ are Lagrangians such that $ L = L ^\prime \circ \Phi $. If $ \Phi $ is a~fiberwise isomorphism, then $ ( a, v , p ) \in \mathcal{P} ( A , A ^\ast ) $ satisfies the Hamilton--Pontryagin principle for $L$ if and only if $ ( a ^\prime , v ^\prime , p ^\prime ) \in \mathcal{P} ( A ^\prime , A ^{\prime \ast} ) $ satisfies the Hamilton--Pontryagin principle for $L ^\prime $, where $ a ^\prime = \Phi \circ a $, $ v ^\prime = \Phi \circ v $, and $ p ^\prime = (\Phi ^\ast) ^{-1} \circ p $.
\end{Theorem}

\begin{proof}
This can be shown directly from the variational principle~-- observing that admissible variations in $ \mathcal{F} _{ ( a, v, p ) } ( A , A ^\ast ) $ map to those in $ \mathcal{F} _{ ( a ^\prime , v ^\prime , p ^\prime ) } ( A ^\prime , A ^{ \prime \ast } ) $, and vice versa~-- but we give an equivalent proof using the implicit Euler--Lagrange--Poincar\'e equations together with the Weinstein--Mart\'inez reduction theorem (Theorem~\ref{thm:weinsteinMartinez}).

First, since $ \Phi $ is a f\/iberwise isomorphism, we have $ a = v $ if and only if $ a ^\prime = v ^\prime $. Moreover, since $ L = L ^\prime \circ \Phi $, the following diagram commutes:
 \begin{equation*}
 \begin{tikzcd}
 A \ar[r, "\Phi", "\cong"'] \ar[d, "\mathrm{d} L ^\text{ver}"'] &
 A ^\prime \ar[d, "\mathrm{d} L ^{\prime \text{ver}}"]\\
 A ^\ast & \ar[l, "\Phi ^\ast ", "\cong"'] A ^{ \prime \ast }.
 \end{tikzcd}
 \end{equation*}
 It follows from this that $ p = \mathrm{d} L ^\text{ver} (v) $ if and only if $ p ^\prime = \mathrm{d} L ^{ \prime \text{ver}} ( v ^\prime ) $. Finally, substituting these expressions for $v$ and $p$ into the f\/irst equation in~\eqref{eqn:implicitELP}, we have
 \begin{gather*}
 \rho ^\ast \mathrm{d} L ^\textup{hor} (a) + \overline{ \nabla }_a ^\ast \mathrm{d} L ^\textup{ver} (a) = 0 , \qquad \rho ^{\prime \ast} \mathrm{d} L ^{\prime \textup{hor}} (a ^\prime ) + \overline{ \nabla } _{a ^\prime } ^{^\prime \ast} \mathrm{d} L ^{\prime \textup{ver}} (a ^\prime ) = 0 .
 \end{gather*}
But these are just the Euler--Lagrange--Poincar\'e equations \eqref{eqn:connectionELP} for $L$ and $ L ^\prime $, respectively. So Theorem~\ref{thm:weinsteinMartinez} implies that one holds if and only if the other does.
\end{proof}

Fortunately, the f\/iberwise isomorphism assumption is still suf\/f\/icient to perform reduction when $ A = V Q \rightarrow Q $ and $ A ^\prime = V Q / G \rightarrow Q / G $, since the quotient map for the groupoid action in Lemma~\ref{lem:quotientMorphism} is a f\/iberwise isomorphism. (Indeed, Higgins and Mackenzie~\cite{HiMa1993} refer to Lie algebroid morphisms with this property as \emph{action morphisms}.) Intuitively, this is because the quotient is taken both on the total space and on the base, so the dimension of the f\/ibers remains the same.

\begin{Theorem} \label{thm:hpGroupoidReduction}
 Let $ G \rightrightarrows M $ be a Lie groupoid and $ Q \rightarrow M $ a principal $G$-space. Suppose the Lagrangian $ L \colon V Q \rightarrow \mathbb{R} $ is $G$-invariant, i.e., that it factors through the quotient morphism $ \Phi \colon V Q \rightarrow V Q / G $ as $ L = \ell \circ \Phi $, where $ \ell \colon V Q / G \rightarrow \mathbb{R} $ is called the reduced Lagrangian. Then $ ( a, v, p ) \in \mathcal{P} ( V Q , V ^\ast Q ) $ satisfies the Hamilton--Pontryagin principle for $L$ if and only if $ ( a ^\prime , v ^\prime , p ^\prime ) \in \mathcal{P} ( V Q / G, V ^\ast Q / G ) $ satisfies the Hamilton--Pontryagin principle for $ \ell $, where $ a ^\prime = \Phi \circ a $, $ v ^\prime = \Phi \circ v $, and $ p ^\prime = (\Phi ^\ast) ^{-1} \circ p $.
\end{Theorem}

\begin{proof} Apply Theorem~\ref{thm:hpMorphismReduction} to the quotient morphism~$\Phi$, which is a f\/iberwise-isomorphic Lie algebroid morphism from $ V Q $ to $ V Q / G $.
\end{proof}

\begin{Example}As in Example~\ref{ex:groupActionReduction}, when $ G \rightrightarrows \bullet $ is a Lie group acting freely and properly on $ Q \rightarrow \bullet $, this corresponds to the case of ordinary Lagrangian reduction for the Hamilton--Pontryagin principle. In the special case where $ Q = G $ acts on itself by multiplication, this gives Euler--Poincar\'e-type reduction for the Hamilton--Pontryagin principle, as in Yoshimura and Marsden~\cite{YoMa2007}, Bou-Rabee and Marsden~\cite{BoMa2009}.
\end{Example}

\begin{Example}As in Example~\ref{ex:groupoidSelfReduction}, suppose $ G \rightrightarrows M $ is a Lie groupoid acting on itself by multiplication, so that the quotient morphism is $ \Phi \colon V ^\beta G \rightarrow V ^\beta G / G = A G $. If $ L \colon V ^\beta G \rightarrow \mathbb{R} $ and $ \ell \colon A G \rightarrow \mathbb{R} $ are Lagrangians satisfying $ L = \ell \circ \Phi $, then Theorem~\ref{thm:hpGroupoidReduction} implies that the implicit vertical Euler--Lagrange equations \eqref{eq:implicitVEL} on $ V ^\beta G $ reduce to the implicit Euler--Lagrange--Poincar\'e equations \eqref{eqn:implicitELP} on $ A G $. The even more special case where $ G \rightrightarrows \bullet $ is a Lie group again gives Hamilton--Pontryagin reduction from $G$ to $ \mathfrak{g} $, as in Yoshimura and Marsden~\cite{YoMa2007}, Bou-Rabee and Marsden~\cite{BoMa2009}.
\end{Example}

\subsection*{Acknowledgments and disclosures}
The authors wish to thank Rui Loja Fernandes for his helpful feedback on this work. This paper also benef\/ited substantially from the suggestions of the anonymous referees, to whom we wish to express our sincere gratitude.
This research was supported in part by grants from the Simons Foundation (award 279968 to Ari Stern) and from the National Science Foundation (award DMS~1363250 to Xiang Tang).
The authors declare that they have no conf\/lict of interest.

\pdfbookmark[1]{References}{ref}
\LastPageEnding

\end{document}